\documentclass[11pt]{amsart}
\usepackage{graphicx,indentfirst}
\usepackage{amsmath,amssymb,mathrsfs}
\usepackage{amsthm,amscd}
\usepackage{verbatim}
\usepackage{appendix}
\usepackage{enumitem,titletoc}
\usepackage{appendix}
\usepackage{fancyhdr}
\usepackage{amsfonts,color}
\usepackage[all]{xy}
\usepackage{syntonly}
\usepackage{fancyhdr}
\newtheorem{theorem}{Theorem}[section]

\newtheorem{conjecture}{Conjecture}[section]
\newtheorem{lemma}{Lemma}[section]
\newtheorem{remark}{Remark}[section]
\usepackage[left=2.5cm,right=2.5cm,bottom=3cm,top=3cm]{geometry}

\newcommand{\abs}[1]{|#1|^2}

\usepackage{tikz}

\def\XXint#1#2#3{{\setbox0=\hbox{$#1{#2#3}{\int}$ }
\vcenter{\hbox{$#2#3$ }}\kern-.6\wd0}}

\newtheorem{prop}{Proposition}[section]
\newtheorem{defn}{Definition}[section]
\newtheorem{corr}{Corollary}[section]

\allowdisplaybreaks

\newcommand{\ric}{\mathrm{Ric}}

\newcommand{\bk}[1]{\Big(#1\Big)}
\newcommand{\ddbar}{i\partial\bar\partial}

\newcommand{\cp}{\mathbb{CP}}

\numberwithin{equation}{section}
\begin{document}
\title[On Feldman-Ilmanen-Knopf conjecture for K\"ahler Ricci flow]{On Feldman-Ilmanen-Knopf conjecture for the blow-up behavior of the K\"ahler Ricci flow}

\author[B. Guo]{Bin Guo$^*$}
\thanks{Research supported in part by NSF grants DMS-1406124 and DMS-1406164}
\address{$^*$ Department of Mathematics, Rutgers University, Piscataway, NJ 08854}\email{bguo@math.rutgers.edu}
\author[J. Song]{Jian Song$^\dagger$}
\address{$^\dagger$ Department of Mathematics, Rutgers University, Piscataway, NJ 08854}\email{jiansong@math.rutgers.edu}

\maketitle
\begin{abstract}  We consider the Ricci flow on $\cp^n$ blown-up at one point starting with any $U(n)$-invariant K\"ahler metric. It is proved in \cite{Zhux, Fo, Song1} that the K\"ahler-Ricci flow must develop Type I singularities. We show that  if the total volume does not go to zero at the singular time, then any Type I parabolic blow-up limit of the Ricci flow along the exceptional divisor is the unique $U(n)$-complete shrinking K\"ahler-Ricci soliton  on $\mathbb C^n$ blown-up at one point.  This establishes the conjecture of Feldman-Ilmanen-Knopf \cite{FIK}.
\end{abstract}

\section{Introduction}

The Ricci flow, first introduced by Hamilton (\cite{H}), is the parabolic equation
$$\frac{\partial g}{\partial t}  = -2\ric(g)$$ evolving the Riemannian metrics by its Ricci curvature. It has become a fundamental tool to study geometry and topology. The K\"ahler-Ricci flow is the Ricci flow on a K\"ahler manifold  starting with a K\"ahler metric. The K\"ahler Ricci flow has developed into a vast field and has made important progress in recent years (e.g.\cite{Pe,SeTi,PS,PSS,PSSW1,PSSW2,PSSW0,ST1,ST0,ST3,TiZz,ChWa} this list is far from complete).

In this paper, we study the unnormalized K\"ahler Ricci flow 
\begin{equation}\label{KRF}
\frac{\partial \omega}{\partial t}  = -\ric(\omega), \quad \omega(0) = \omega_0
\end{equation}
on $X = \cp^n\# \overline{\cp^n}$, i.e., $\cp^n$ blown-up at one point. We will always assume that the initial K\"ahler metric $\omega_0$ is invariant  under the action of a maximal compact subgroup $U(n)$ of the automorphism group of $X$. It is proved (\cite{SWe0}) that the flow \eqref{KRF} must develop finite time singularity and it either shrinks to a point, collapses to $\cp^{n-1}$ or contracts an exceptional divisor, in the Gromov-Hausdorff topology.

The Ricci flow solution $g(t)$ is said to develop type I singularity on $X$ at the finite singular time $T$ if there exists $C>0$ such that 
$$ \sup_{X\times [0, T)} (T-t) |Rm(g(t))| \leq C. $$
It is proved in \cite{Nab, EMT} that if the Ricci flow develops type I singularity on a closed manifold, then the type I blow-up limit along essential singularities must be a nontrivial complete shrinking Ricci soliton.

$\mathbb{CP}^n$ blown-up at one point  is in fact a $ \mathbb{CP}^1$ bundle over $ \mathbb{CP}^{n-1}$ given by
$$X= \mathbb{P}( \mathcal{O}_{ \mathbb{CP}^{n-1}}\oplus \mathcal{O}_{ \mathbb{CP}^{n-1}}(-1)).$$
Let $D_0$ be the exceptional divisor of $X$ defined by the image of the section $(1,0)$ of $\mathcal{O}_{\mathbb{CP}^{n-1}}\oplus \mathcal{O}_{\mathbb{CP}^{n-1}}(-1)$ and $D_\infty$ be the divisor of $X$ defined by the image of the section $(0,1)$ of $\mathcal{O}_{\mathbb{CP}^{n-1}}\oplus \mathcal{O}_{ \mathbb{CP}^{n-1}}(-1)$. Both the $0$-section $D_0$ and the $\infty$-section are complex hypersurfaces in $X$ isomorphic to $ \mathbb{CP}^{n-1}$.  The K\"ahler cone on $X$ is given by
$$\mathcal{K}= \{ -a[D_0] + b [D_\infty] ~|~ 0<a<b\}.$$ In particular, when $n=2$, $D_0$ is a holomorphic $S^2$ with self-intersection number $-1$. We will write  the exceptional divisor of $X$ as $E$ and it is in fact equal to $D_0$.

Let $\omega_0$ be the initial $U(n)$ invariant K\"ahler metric of the Ricci flow (\ref{KRF}) on $X$. We let $[\omega_0] \in 
b_0 [D_\infty] - a_0 [D_0] $ with $0<a_0<b_0$. Then the limiting behavior of the Ricci flow can be summarized in the following three cases. 

When the initial K\"ahler class is proportional to the first Chern class, i.e.
$$ a_0 (n-1) = b_0 (n+1),$$
 the flow shrinks to a point at the singular time $T= a_0/(n-1)$ (\cite{SWe0}).
It is shown in  \cite{Zhux} that the flow must develop Type I singularities and the rescaled Ricci flow converges in the Cheeger-Gromov-Hamilton sense to the unique compact shrinking K\"ahler Ricci soliton on $X$ constructed in \cite{C1,Koiso,WZ}.

When the initial K\"ahler class satisfies%
$$  a_0 (n-1) > b_0 (n+1),$$
 the flow collapses to $\cp^{n-1}$ at $T= (b_0-a_0)/2$ (\cite{SWe0}). It is shown in \cite{Fo} that the flow must develop Type I singularities and the rescaled flow converges in Cheeger-Gromov-Hamilton sense to the ancient solution that splits isometrically as $\mathbb C^{n-1}\times \cp^1$.

The initial K\"ahler class condition of %
$$  a_0 (n-1) < b_0 (n+1),$$
is equivalent to  the limiting total volume being strictly positive  at the singular time $T =  a_0/(n-1)$, i.e., 
\begin{equation}\label{volnot}
\liminf_{t\to T^-}  Vol(X,g(t))>0,
\end{equation} and 
the flow contracts the exceptional divisor $D_0$ at $T$ (\cite{SWe0}). In fact $  a_0 (n-1) < b_0 (n+1)$ is equivalent to the condition (\ref{volnot}) . It is then shown in \cite{Song1} that the flow \eqref{KRF} must develop Type I singularities and the parabolic blow up of the Type I Ricci flow along the exceptional divisor converges to a complete non-flat shrinking K\"ahler Ricci soliton on a complete manifold {\em diffeomorphic} to $\mathbb C^n$ blown-up at one point.

\begin{theorem}[\cite{Song1}]
Let $X$ be $\cp^n$ blown-up at one point and $E$ be the exceptional divisor. Let $g(t)$ be the $U(n)$-invariant solution to \eqref{KRF} on $X$ on $[0,T)$, where $T\in (0, \infty)$ is the singular time of the flow. If 
\begin{equation*}
\liminf_{t\to T^-}  Vol(X,g(t))>0,
\end{equation*}
 the flow develops type I singularity. 
Moreover, for any sequence $t_j\to T$, we consider the type I parabolic rescaled flows $(X,p, g_j(t))$ defined on $[-\frac{t_j}{T-t_j}, 1)$ by
\begin{equation}\label{scaling metrics}g_j(t) = \frac{1}{T-t_j} g(t_j + t (T-t_j))\end{equation} with a fixed base point $p\in E$.  
Then  there exist a subsequence converging in  Cheeger-Gromov-Hamilton  sense ($C^\infty$-topology) to a complete shrinking non-flat gradient K\"ahler Ricci soliton on a complete K\"ahler manifold {\em diffeomorphic} to $\mathbb C^n$ blown-up at one point.

\end{theorem}

 It is proved by Feldman-Ilmanen-Knopf \cite{FIK}) that there exists a unique $U(n)$ invariant complete K\"ahler-Ricci gradient shrinking soliton on $\mathbb{C}^n$ blown-up at one point (FIK soliton) and they further made the following conjecture. 

\begin{conjecture}\label{conjecture} Let $g(t)$ be the $U(n)$ invariant metrics satisfying the K\"ahler-Ricci flow on $X=\mathbb{CP}^n$ blown-up at one point for $t\in [0, T)$. Let $T\in (0, \infty)$ be the singular time and $$\liminf_{t\rightarrow T^-} Vol(X, g(t)) >0. $$ Then the flow develops type I singularities and any type I parabolic blow-up limit of $g(t)$ with a fixed base point in the exceptional divisor $E$ is the unique FIK soliton on $\mathbb{C}^n$ blown-up at one point. 

\end{conjecture}

This conjecture was partially established by Maximo (\cite{Ma}) when the dimension $n=2$ under certain open conditions on the initial metric.  Our main result in this paper is to show that in the non-collapsed case, the blow-up limit of the K\"ahler Ricci flow is {\em biholomorphic} to $\mathbb C^n$ blown-up at one point and the limit K\"ahler Ricci soliton is the FIK soliton constructed in \cite{FIK} on $\mathbb C^n$ blown-up at one point, hence establishing Conjecture \ref{conjecture}. Our main theorem is
%
\begin{theorem}\label{main theorem}Let $X$ be $\cp^n$ blown-up at one point and $E$ be the exceptional divisor. Let $g(t)$ be the $U(n)$-invariant solution to \eqref{KRF} on $X$ on $[0,T)$. If 
\begin{equation*}
\liminf_{t\to T^-}  Vol(X,g(t))>0,
\end{equation*}
we fix any base point $p\in E$ and  let $(X_\infty, p_\infty, g_\infty)$ be the Cheeger-Gromov-Hamilton limit of $(X,p, g_j(t))$, where $g_j(t)$ is defined by \eqref{scaling metrics}. Then $(X_\infty, p_\infty, g_\infty)$  is biholomorphic to $\mathbb C^n$ blown-up at one point and $g_\infty$ is a complete, $U(n)$ symmetric K\"ahler Ricci soliton metric, hence is one of the FIK solitons constructed in \cite{FIK}. 
\end{theorem}
For $n\geq 2$, there exist infinitely many distinct complex structures on $\mathbb{R}^{2n}$ and so on its complex blow-up at a point. $U(n)$ symmetry in the complex setting is more complicated than $O(2n)$ symmetry in the real setting due to the complex structures, in particular, the complex structures might possibly degenerate or jump the variation limits.   For example, the manifolds $\mathcal O_{\cp^{n}}(-k)$ with odd $1\le k<n$ are all  diffeomorphic, but as complex manifolds they admit different complex structures and hence different $U(n)$-invariant complete shrinking K\"ahler Ricci soliton metrics (\cite{FIK}). Our strategy is (1) to construct a $U(n)$-action on the limit manifold $X_\infty$, which is holomorphic with respect to the limit complex structure on $X_\infty$, (2) to construct a holomorphic fiber bundle map $F_\infty:X_\infty\to \cp^{n-1}$, and (3) to show this fiber bundle is in fact $\mathcal O_{\cp^{n-1}}(-1)$ and the limit metric $g_\infty$ is $U(n)$-invariant. Our proof can also be applied to the K\"ahler-Ricci flow on $\mathbb{P}( \mathcal{O}_{\mathbb{CP}^{n-1}} \oplus \mathcal{O}_{\mathbb{CP}^{n-1}}(-k))$ with $U(n)$-invariant initial K\"ahler metric for $1\leq k \leq n-1$ and as long as the total volume does not tend to $0$ at the singular time, the flow must develop type I singularities and the type I blow-up limit along the exceptional divisor must be the unique $U(n)$-invariant complete shrinking gradient K\"ahler-Ricci soliton on $\mathcal{O}_{\mathbb{CP}^{n-1}}(-k)$ constructed in \cite{FIK}.

This paper is organized as follows. In Section \ref{section 2}, we collect some known facts about the K\"ahler Ricci flow with $U(n)$ symmetry on $X$. In Section \ref{section 3}, we prove some a priori estimates and construct the limit map $F_\infty: X_\infty\to \cp^{n-1}$ and limit holomorphic vector field $V_\infty$ on $X_\infty$. In Section \ref{section 4}, we show that the $U(n)$ actions on $X_j$ can pass to the limit $X_\infty$ and we can define a $U(n)$-action on $X_\infty$, and prove that $X_\infty$ is either the holomorphic line bundle $\mathcal O_{\cp^{n-1}}(-1)$ or the disk subbundle of $\mathcal O_{\cp^{n-1}}(-1)$. In Section \ref{section 5} we finish the proof of Theorem \ref{main theorem} by showing that if $X_\infty$ is the disk bundle in $\mathcal O_{\cp^{n-1}}(-1)$ then the limit metric $g_\infty$ cannot be complete, hence $X_\infty$ is $\mathcal O_{\cp^{n-1}}(-1)$.

\medskip
Throughout this paper, we will use $\omega$ to denote the K\"ahler form of a K\"ahler metric $g$, without specifically mentioning this. And  $C$ will denote a uniform constant depending only on the dimension $n$ and the initial K\"ahler metric, which may be different from line to line. 

\bigskip

\noindent{\bf Acknowledgment:} We would like to thank  Xiaowei Wang for many helpful discussions.

\section{Preliminaries}\label{section 2}
In this section, we collect some backgrounds and known results about the flow \eqref{KRF}.
\subsection{Calabi symmetry}Let $X=\cp^n\#\overline{\cp^n}$ be $\mathbb{CP}^n$ blown-up at one point and it is a $\mathbb{CP}^1$ bundle over $\cp^{n-1}$ given by
\begin{equation}\label{fiber structure}X = \mathbb P(\mathcal O_{\cp^{n-1}}\oplus \mathcal O_{\cp^{n-1}}(-1)).\end{equation}
Let $D_0$ be the exceptional divisor of $X$ defined by the image of the section $(1,0)$ of $\mathcal O_{\cp^{n-1}}\oplus \mathcal O_{\cp^{n-1}}(-1)$ and $D_\infty$ be the divisor  defined by the image of the section $(0,1)$ of $\mathcal O_{\cp^{n-1}}\oplus \mathcal O_{\cp^{n-1}}(-1)$. Both divisors $D_0$ and $D_\infty$ are complex hypersurfaces isomorphic to $\cp^{n-1}$. The K\"ahler cone on $X$ is given by  $$\mathcal K = \{-a[D_0] + b[D_\infty]~|~0<a<b\}.$$
Let $z = (z_1,\ldots,z_n)$ be the standard complex coordinates on $\mathbb C^n$. Define $\rho = \log \abs{z} = \log(\abs{z_1}+\cdots +\abs{z_n})$ on $\mathbb C^n\backslash\{0\}$. 
\begin{defn}
A smooth convex function $u = u(\rho)$ for $\rho\in (-\infty,\infty)$ is said to satisfy  the {\em Calabi symmetry conditions}, if
\begin{enumerate}[label=(\arabic*)]
\item $u''(\rho)>0, u'(\rho)>0$ for $\rho\in (-\infty,\infty)$,
\item There exist $0<a<b$ and smooth functions $U_0, U_\infty: [0,\infty)\to \mathbb R$ such that $$U_0'(0)>0, \quad U_\infty'(0)>0,$$
$$u(\rho) = a\rho + U_0(e^\rho) \quad \text{near }\rho = -\infty,$$
$$u(\rho) = b\rho + U_\infty(e^{-\rho})\quad \text{near }\rho = +\infty.$$
\end{enumerate}
\end{defn}
It is known (\cite{Calabi}) that a metric $\omega = \ddbar u$ which defines a smooth K\"ahler metric on $\mathbb C^n\backslash\{0\}$  extends to a K\"ahler metric on $X=\cp^n\# \overline{\cp^n}$ if and only if $u$ satisfies the Calabi symmetry condition, and it defines a K\"ahler metric in the class $-a[D_0] + b[D_\infty]$.

On $\mathbb C^n\backslash  \{0\}$, the K\"ahler metric $\omega = \ddbar u$ is given by
\begin{equation}\label{expansion of omega}
\omega =\sqrt{-1}g_{i\bar j}dz_i\wedge d\bar z_j =   \bk{e^{-\rho} u'\delta_{ij} + e^{-2\rho} \bar z_i z_j(u''-u')} \sqrt{-1}dz_i\wedge d\bar z_j.
\end{equation}
The metric $\omega$ is invariant under the standard unitary $U(n)$-actions on $\mathbb C^n$, hence also invariant under the induced $U(n)$-actions on $X$, i.e. $U(n)\subset \mathrm{Isom}(X,\omega)$, the isometry group of $\omega$.

On $\mathbb C^n\backslash\{0\}$, $\det (g_{i\bar j}) = e^{-n\rho} (u')^{n-1} u''$ and the Ricci potential of $\omega=\ddbar u$ is 
$$v = -\log \det g_{i\bar j} = n\rho  - (n-1)\log u' - \log u'',$$ and Ricci curvature tensor of $\omega$ is given by
$$R_{i\bar j} = e^{-\rho} v' \delta_{ij} + e^{-2\rho}\bar z_i z_j (v'' - v').$$
It is known (\cite{SWe0}) that the Calabi symmetry is preserved by the K\"ahler Ricci flow \eqref{KRF}, in other words, the evolving K\"ahler metrics $\omega(t)$ of \eqref{KRF} is invariant under $U(n)$-action if the initial metric $\omega_0$ is $U(n)$-invariant. In \cite{SWe0} it is shown that \eqref{KRF} can be reduced to the following parabolic equation for $u = u(\rho,t)$
\begin{equation}\label{scalar KRF}
\frac{\partial}{\partial t} u(\rho,t) = \log u''(\rho,t) + (n-1)\log u'(\rho,t) - n\rho,
\end{equation}
where the evolving metrics $\omega(t)$ are given by $\omega(t) = \ddbar u(\rho,t)$. If the initial K\"ahler metric $\omega(0)\in -a_0[D_0] + b_0 [D_\infty]$, then the evolving K\"ahler class is given by
$$\omega(t)\in -a_t [D_0] + b_t [D_\infty], \quad \text{with }a_t = a_0 - (n-1) t, ~ b_t = b_0 - (n+1)t.$$

We will identify the zero section $D_0\subset X$ as the exceptional divisor  $E\cong\cp^{n-1}$ in $\mathbb C^n$ blown-up at the origin, and $\mathbb C^n\subset\cp^n$. Under the $U(n)$ invariant metric $g=\omega = \ddbar u$, the distance from a point $z\in \mathbb C^{n}\backslash\{0\}$ to $E$ is given by
$$d_g( z, E) = \frac{1}{2}\int_{-\infty}^{\log \abs{z}} \sqrt{u''(\rho)}d\rho.$$
The Calabi symmetry condition (2) above implies this distance is finite for finite $z(\neq 0)$.

We define the tubular neighborhood  $B_g(E,R)$ of $E$ (in the following we also call $B_g(E,R)$ as metric balls centered at $E$) as
$$B_g(E,R):=\{q\in X~|~ d_g(q,E)\le R\},$$ which (for $R$ small) can be identified as $\pi^{-1}(B)$ for some Euclidean ball $B\subset \mathbb C^n$ centered at $0$ and $\pi: \widetilde{\mathbb C^n}\to \mathbb C^n$ is the blown-up map of $\mathbb C^n$ at $0$. The volume of $B_g(E,R)$ with respect to the metric $\omega = \ddbar u$ is given by
\begin{equation}\label{volume formula}\int_{B(E,R)}\omega^n = C(n)\int_{-\infty}^{\rho_R} (u'(\rho))^{n-1} u''(\rho)d\rho,\end{equation}
for some constant $C(n)$ depending only on the dimension and $\rho_R$ is the unique constant determined by the equation
\begin{equation}\label{rho R}R = \frac 12 \int_{-\infty}^{\rho_R} \sqrt{u''(\rho)}d\rho,\end{equation} i.e., a point $z\in \mathbb C^n\backslash \{0\}$ with $\log \abs{z}= \rho_R$ satisfies $z\in \partial B_g(E,R)$.

We recall the following formulas of gradient and Laplacian of a $U(n)$ invariant function, which follow from direct calculations so we omit the proof.
\begin{lemma}\label{gradient and Laplacian}
Suppose $f$ is a $U(n)$-invariant function on $X$, then with respect to the metric $\omega = \ddbar u$, we have
\begin{equation*}
\abs{\nabla f}_{\omega} = \frac{(f')^2}{u''},\quad \Delta_\omega f = (n-1)\frac{f'}{u'} + \frac{f''}{u''},
\end{equation*}
where as usual for the function $f$, $f' = \frac{\partial}{\partial \rho}f$, $f'' = \frac{\partial^2 f}{\partial \rho^2}$.
\end{lemma}
\subsection{Type I solutions}
Recall the Ricci flow \eqref{KRF} is said to develop {\em Type I singularity} if 
$$\sup_{(x,t)\in X\times [0,T)} (T-t)|Rm|(x,t)<\infty,$$ where $T\in (0,\infty)$ is the singular time.

\begin{theorem}[\cite{Song1, Fo, Zhux}] \label{Song theorem}
Let $X$ be $\cp^n$ blown-up at one point. Then the K\"ahler Ricci flow \eqref{KRF} on $X$ must develop Type I singularities  for any $U(n)$ invariant initial K\"ahler metric. 

Let $g(t)$ be the solution on $[0,T)$. For any $t_j\to T$, we consider the rescaled flows $(X,g_j(t))$ defined on $[-\frac{t_j}{T-t_j}, 1)$ by
$$g_j(t) = \frac{1}{T-t_j} g(t_j + t (T-t_j)).$$
Then one and only one of the following must occur.
\begin{enumerate}[label=(\arabic*)]
\item (\cite{Song1}) If $\liminf_{t\to T} (T-t)^{-1} Vol(X,g(t)) = \infty$, then $(X,g_j(t), p)$ sub-converges in $C^\infty$ Cheeger-Gromov-Hamilton (CGH) sense to a complete shrinking non-flat gradient K\"ahler Ricci soliton on a complete K\"ahler manifold {\em diffeomorphic} to $\widetilde {\mathbb C^n}$, for any fixed point $p\in E$, the exceptional divisor.
\item (\cite{Fo}) If $\liminf_{t\to T}(T-t)^{-1} Vol(X,g(t))\in (0,\infty)$, then $(X,g_j(t),p_j)$ sub-converges in $C^\infty$-CGH sense to $(\mathbb C^{n-1}\times \mathbb {CP}^1, g_{\mathbb C^{n-1}}\oplus (-t)g_{FS})$, where $g_{\mathbb C^{n-1}}$ is the standard flat metric on $\mathbb C^{n-1}$ and $g_{FS}$ is the Fubini-Study metric on $\cp^1$ for any sequence of points $p_j$.

\item(\cite{Zhux}) If $\liminf_{t\to T}(T-t)^{-1} Vol(X,g(t)) = 0$, then $(X,g_j(t))$ converges in $C^\infty$-CGH sense to the unique compact shrinking K\"ahler Ricci soliton on $\cp^n$ blown-up at one point.

\end{enumerate}
\end{theorem}

Our main result in this paper is to show the limit K\"ahler Ricci soliton in case (1) is in fact one of  the FIK solitons constructed in  \cite{FIK}, and the limit space is {\em biholomorphic} to $\widetilde {\mathbb C^n}$, $\mathbb C^n$ blown-up at one point. 

Suppose the initial $U(n)$-invariant K\"ahler metric lies in the class $-a_0[D_0] + b_0[D_\infty]$. It is proved (\cite{SWe0}) that the condition in case (1) above that $\liminf_{t\to T^-} (T-t)^{-1} Vol(X,g(t)) = \infty$ (see also \eqref{volnot}) is equivalent to the inequality $$0<a_0(n+1)<b_0(n-1).$$ 
And the K\"ahler Ricci flow \eqref{KRF} will contract the exceptional divisor $D_0$ at the singular time \begin{equation}\label{singular time}T=\frac{a_0}{n-1}.\end{equation} Throughout this paper we will assume $0<a_0(n+1)<b_0(n-1)$.

\subsection{Cheeger-Gromov convergence}\label{CG convergence}
Let $g_j := g_j(0) = \frac{1}{T-t_j} g(t_j)$ and $X_j = X, p_j = p\in D_0=E$ be a fixed point, then from  case (1) in Theorem \ref{Song theorem}, we know the pointed manifolds $(X_j, p_j,g_j)$ converge in $C^\infty$ Cheeger-Gromov (CG) sense to a complete K\"ahler manifolds $(X_\infty, p_\infty,g_\infty)$ and $g_\infty$ is a nontrivial complete shrinking K\"ahler Ricci soliton. Recall the CG convergence means that there exists a sequence of increasing relatively compact exhaustion $\{U_j\}$ of $X_\infty$, and diffeomorphisms (onto its image) $\phi_j: U_j\to X_j$ satisfying $\phi_j(p_\infty) = p_j$ and 
\begin{equation}\label{smooth convergence}
\phi_j^*g_j\xrightarrow{C^\infty_{loc}} g_\infty,\quad \phi_j^*J_j\xrightarrow{C^\infty_{loc}} J_\infty,
\end{equation}
where $J_j, J_\infty$ are the complex structures on $X_j,X_\infty$, respectively, compatible with the K\"ahler metrics $g_j, g_\infty$.

Since the restriction of the metrics $g_j$ to $E$ are $(n-1) g_{FS}$ where $g_{FS}$ is the Fubini-Study metric on $\cp^{n-1}$, we have
\begin{lemma}[see also \cite{Song1}]\label{diameter bound} 
The diameter of $(E,g_j|_E)$ is $D_n=\alpha_n (n-1)^{1/2}$, hence uniformly bounded. Here $\alpha_n =$ the diameter of $(\cp^{n-1},g_{FS})$.
\end{lemma}

For notational convenience, we will also denote the exceptional divisor $E\subset X_j = X$ by $E_j$.

\section{A priori estimates}\label{section 3}

As we mentioned before, we will assume the initial K\"ahler metric lies in $-a_0[D_0] + b_0[D_\infty]$ with $0<a_0(n+1)<b_0(n-1)$. The evolving metrics belong to the K\"ahler classes
$$\omega(t)\in -a_t [D_0] + b_t [D_\infty], \quad \text{with }a_t = a_0 - (n-1) t, ~ b_t = b_0 - (n+1)t.$$ The evolution equations for the potentials of the evolving metrics $\omega(t) = \ddbar u(\rho, t)$ for $\rho\in (-\infty,\infty)$ and $t\in [0,T)$, where $T$ is given in \eqref{singular time}, are given by (\cite{SWe0,Song1})
\begin{equation}\label{prime equation}
\frac{\partial}{\partial t} u' = \frac{u'''}{u''} + (n-1)\frac{u''}{u'} - n,
\end{equation}
\begin{equation}\label{double prime equation}
\frac{\partial}{\partial t} u'' = \frac{u^{(4)}}{u''} - \frac{(u''')^2}{(u'')^2} + (n-1)\frac{u'''}{u'} - (n-1)\frac{(u'')^2}{(u')^2},
\end{equation}
\begin{equation}\label{triple prime equation}\begin{split}
\frac{\partial}{\partial t} u''' = &\frac{u^{(5)}}{u''} - \frac{3u''' u^{(4)}}{(u'')^2} + \frac{2(u''')^3}{(u'')^3} +(n-1)\frac{u^{(4)}}{u'}\\
&- 3(n-1)\frac{u'' u'''}{(u')^2} + 2(n-1) \frac{(u'')^3}{(u')^3}.
\end{split}\end{equation}
Along the flow \eqref{KRF} or \eqref{scalar KRF}, we have (see (\cite{SWe0, Song1}))
\begin{lemma}\label{elementary bound}
There exists a constant  $C>0$ such that for all $t\in [0,T)$ and $\rho\in (-\infty,\infty)$ such that
\begin{equation}\label{u prime bound}
(n-1)(T-t) =a_t\le u'\le C,
\end{equation}
and 
\begin{equation}\label{ratio bound}
0\le \frac{u''}{u'}\le C, \quad -C\le \frac{u'''}{u''}\le C.
\end{equation}
\end{lemma}

\begin{lemma}\label{bound of u prime}
There exists a constant $C>0$ such that for all $t\in [0,T)$ and $\rho\in (-\infty,\infty)$ 
\begin{equation}\label{bound of u double prime}
C^{-1} (u'-a_t)(b_t - u')\le u''\le C (u'-a_t)(b_t - u').
\end{equation}
\end{lemma}
\begin{proof}
The proof of the second inequality is given in Lemma 4.5 of \cite{SWe0}, and the first inequality can be proved following the same argument as in \cite{SWe0}. For readers' convenience we include the proof below.

Consider the quantity $H = \log u'' - \log (u'-a_t) - \log (b_t - u')$, using the evolution equations \eqref{prime equation} and \eqref{double prime equation} we have
\begin{equation}\label{H evolution}\begin{split}
\frac{\partial H}{\partial t} = & \frac{1}{u''}\bk{ \frac{u^{(4)}}{u''} - \frac{(u''')^2}{(u'')^2} + (n-1)\frac{u'''}{u'} - (n-1) \frac{(u'')^2}{(u')^2} }\\
& - \frac{1}{u'-a_t} \bk{ \frac{u'''}{u''} + (n-1)\frac{u''}{u'} - 1 } - \frac{1}{b_t - u'}\bk{ -\frac{u'''}{u''} - (n-1)\frac{u''}{u'} - 1 }.
\end{split}\end{equation}
It can be checked by Calabi symmetry condition that for each fixed $t\in [0,T)$
\begin{equation*}
\lim_{\rho \to \pm \infty}\frac{u''(\rho,t)}{(b_t - u'(\rho,t)) (u'(\rho,t) - a_t)} = \frac{1}{b_t - a_t},
\end{equation*}
which is uniformly bounded above and below in our case.

For any $T'\in (0, T)$, suppose the minimum of $H$ on $X\times [0,T']$ is obtained at some $(\rho_0,t_0)$, then at this point we have $\frac{\partial}{\partial t} H\le 0$, $H' = 0$, and $H''\ge 0$, i.e.
\begin{equation}\label{equation H 1}
\frac{u'''}{u''} - \frac{u''}{u'-a_t} + \frac{u''}{b_t - u'} = 0,
\end{equation}
\begin{equation}\label{equation H 2}
\frac{u^{(4)}}{u''} - \frac{(u''')^2}{(u'')^2} - \frac{u'''}{u'-a_t} + \frac{(u'')^2}{(u'-a_t)^2} + \frac{u'''}{b_t - u'} + \frac{(u'')^2}{(b_t - u')^2}\ge 0,
\end{equation}
combining with \eqref{H evolution}, \eqref{equation H 1} and \eqref{equation H 2} we have at  $(\rho_0,t_0)$,
\begin{equation}
- u''\bk{\frac{1}{(u'-a_t)^2} + \frac{1}{(b_t - u')^2}} - (n-1)\frac{u''}{(u')^2} + \frac{1}{u'-a_t} + \frac{1}{b_t - u'}\le 0.
\end{equation}
Hence
\begin{equation*}
\frac{u''}{(b_t - u')(u'-a_t)} + \frac{(n-1) u'' (b_t - u')(u'-a_t)}{(u')^2 \big((u'-a_t)^2 + (b_t - u')^2 \big)} \ge \frac{b_t -a_t }{(b_t - u')^2 + (u'-a_t)^2}.
\end{equation*}
We observe that 
\begin{equation*}
\frac{(b_t - u')(u'-a_t)}{(u')^2 ((u'-a_t)^2+ (b_t - u')^2)}\le \frac{1}{(b_t - u')(u'-a_t)},
\end{equation*}
and $$(u'-a_t)^2 + (b_t - u')^2\le 2 (b_t -a_t)^2,$$ hence at $(\rho_0,t_0)$,
\begin{equation*}
\frac{u''}{(u'-a_t) (b_t - u')}\ge \frac{1}{2n (b_t - a_t)}\ge C^{-1},
\end{equation*}
as $b_t - a_t$ is uniformly bounded above. The maximum principle implies the minimum of $H$ on $X\times [0,T']$ is uniformly bounded below independent of the choice of $T'$, hence we conclude that $\inf_{X\times [0,T)} H\ge - C$. And we finish the proof the first inequality in \eqref{bound of u double prime}.

\end{proof}

\subsection{Estimates for the sequence of metrics $g_j$}
Recall that the Cheeger-Gromov limit $(X_\infty, g_\infty, p_\infty)$ of $(X_j,g_j,p)$  is a complete K\"ahler Ricci soliton by the work of \cite{Nab, EMT}. Hence by a theorem of Cao-Zhou (Theorem 1.2 in \cite{CZ}), there exists a constant $C_0>0$ such that the volume of geodesic balls $B_{g_\infty}(p_\infty,R)$ satisfies
\begin{equation}\label{volume growth}
Vol_{g_\infty}(B_{g_\infty}(p_\infty,R))\le C_0 R^{2n}.
\end{equation}
And Perelman's non-collapsing (\cite{Pe}) implies there exists a $\kappa>0$ for any $q\in X_\infty$, the volume $Vol_{g_\infty}(B_{g_\infty}(q, r_0))\ge \kappa r_0^{2n}$, for any $r_0^2\le \frac{1}{C}$, where $C$ is the constant in Type I condition. In particular, we have $Vol_{g_\infty}(B_{g_\infty}(p_\infty,R))\to \infty$ as $R\to \infty$.

\begin{lemma}\label{estimate lemma}
For any $R>0$, there exist constants $c(n,R)>0$ and $C(n,R) = O(R^2)$ such that for $j\ge 1$ large enough, then in the metric balls $B_{g_j}(E_j, R)$, we have 
\begin{equation}\label{good bound 1}
(n-1)(T-t_j)= a_{t_j}\le u'(\rho, t_j) \le C(n,R)(T-t_j),\quad u''(\rho, t_j)\le C(n,R)(T-t_j).
\end{equation}
Moreover, on $\partial B_{g_j}(E_j,R)$, for $j$ large enough, we have
\begin{equation}\label{good bound}\begin{split}
c(n,R)(T-t_j) + (n-1)(T-t_j)&\le u'(\rho_{j,R},t_j)\le C(n,R)(T-t_j),\\
c(n,R)(T-t_j)&\le  u''(\rho_{j,R},t_j)\le C(n,R)(T-t_j)
\end{split}\end{equation}
and $c(n,R)\to +\infty$ as $R\to \infty$, $\rho_{j,R}$ is defined in \eqref{rho R}, corresponding to points on $\partial B_{g_j}(E_j, R)$. 
\end{lemma}

\begin{proof}
For any fixed $R>0$, by the $C^\infty$-CG convergence \eqref{smooth convergence} we have 
\begin{equation*}
Vol_{g_j}(B_{g_j}(p_j, R))\to Vol_{g_\infty}(B_{g_\infty}(p_\infty, R)),\quad\text{as }j\to \infty,
\end{equation*}
in particular, we have both $Vol_{g_j}(B_{g_j}(p_j, R))$ and $Vol_{g_j}(B_{g_j}(p_j, R+D_n))$ are uniformly bounded above and below, for $j$ large enough, where $D_n$ is the diameter of $E_j$ given by Lemma \ref{diameter bound}. Noting that $$B_{g_j}(p_j, R)\subset B_{g_j}(E_j, R)\subset B_{g_j}(p_j, R+D_n), $$
hence there are two constants $c_1 = c_1(n,R)$ and $C_1= C_1(n,R)$
\begin{equation}\label{volume bound}
c_1(n,R)\le Vol_{g_j} (B_{g_j}(E_j, R))\le C_1(n,R),
\end{equation}
By \eqref{volume growth}, it is easy to see that when $j$ is large enough, we can choose $C_1(n,R) = O(R^{2n})$. Moreover,
by the volume formula \eqref{volume formula}
\begin{equation}\label{volume expression}\begin{split}Vol_{g_j} (B_{g_j}(E_j, R)) &= \frac{C(n)}{(T-t_j)^n}\int_{-\infty}^{\rho_{j,R}} (u'(\rho,t_j),t_j)^{n-1}u''(\rho,t_j)d\rho\\
&= \frac{C(n)}{n(T-t_j)^{n}} \bk{ (u'(\rho_{j,R},t_j))^n - a_{t_j}^n},
\end{split}\end{equation}
where $a_{t_j} = (n-1)(T-t_j)$, and $\rho_{j,R}$ is a constant determined by the equation \eqref{rho R} with $u''$ replaced by $\frac{u''(\rho,t_j)}{T-t_j}$.
Combining \eqref{volume bound} and \eqref{volume expression}, there are constants $c_2(n,R)$ and $C_2(n,R)$ such that
\begin{equation}\label{u prime bound on balls}
c_2(n,R)+(n-1)\le \frac{u'(\rho_{j,R},t_j)}{T-t_j}\le C_2(n,R).
\end{equation}
Combining with the fact that $u'(\rho,t_j)$ is increasing in $\rho$ and Lemmas \ref{elementary bound} and \ref{bound of u prime}, if $j$ is large enough, \eqref{good bound 1} and \eqref{good bound} hold.\end{proof}

\subsection{$X_j$ as a $\cp^1$-bundle over $\cp^{n-1}$} \label{subsection 3.2}Recall the manifold $X_j$ can be viewed as a $\cp^1$-bundle over $\cp^{n-1}$ (see \eqref{fiber structure}). Let $$F_j: X_j\to \cp^{n-1}$$ be the holomorphic bundle map. 
\begin{lemma}\label{C1 estimate}
The holomorphic maps $F_j: (X_j,\omega_j)\to (\cp^{n-1}, \omega_{FS})$ have uniformly bounded derivatives, i.e., there exists $C>0$ such that for all $j$, 
\begin{equation}
\sup_{X_j} |\nabla F_j|_{\omega_j, \omega_{FS}} \leq C
\end{equation}
Furthermore, the derivatives $dF_j: TX_j \to T\cp^{n-1}$ have full rank $n-1$ everywhere. 
\end{lemma}
\begin{proof}

Note that $F_j^*\omega_{FS} = \ddbar \rho = \ddbar \log \abs{z}$. As a map $F_j:(X_j,\omega_j)\to (\cp^{n-1}, \omega_{FS})$, its energy density $e(F_j): = tr_{\omega_j} f_j^* \omega_{FS} = \Delta_{\omega_j}\rho$ is equal to
\begin{equation*}
\frac{(T-t_j) (n-1)}{u'(\rho,t_j)}\le 1,
\end{equation*}
hence the differential of maps $F_j$, $dF_j : TX_j\to T\cp^{n-1}$ is uniformly bounded.

Moreover, by \eqref{good bound 1}, in the balls $B_{g_j}(E_j,R)$,
\begin{equation}\label{lower energy bound}
e(F_j)\ge \frac{(n-1)}{C(n,R)}.
\end{equation}
By the symmetry of $F_j$ and $\omega_j$, it is not hard to see the $(n-1)$-many nonzero eigenvalues of $\omega_j^{-1}\cdot F_j^* \omega_{FS}$ are bounded below by $\frac{1}{C(n,R)}$ in $B_{g_j}(E_j,R)$. And this implies that the rank of the differential map $dF_j:TX_j\to T\cp^{n-1}$ is $n-1$.
\end{proof}

\begin{lemma}\label{C2 estimate}
There exists a constant $C=C(n)>0$, such that for any $j\ge 1$,
\begin{equation*}
\abs{\nabla\nabla F_j}_{g_j, g_{FS}}\le C.
\end{equation*}
Hence we have uniform $C^2$ bound of the maps $F_j$.
\end{lemma}
\begin{proof}
Since $F_j$ is holomorphic and $\omega_j$ and $\omega_{FS}$ are K\"ahler metrics, $F_j:X_j\to \cp^{n-1}$ is also a harmonic map. By the Bochner formula
\begin{equation}\label{Bochner formula}
\Delta e(F_j) = \abs{\nabla\nabla F_j} + \ric_{\omega_j}(\nabla F_j, \overline{\nabla F_j}) - (R_{\omega_{FS}})^\alpha_{~\beta \bar\gamma\delta} \overline{(F_j)^\alpha_i} (F_j)^\delta_i (F_j)^\beta_{k}\overline{(F_j)^\gamma_k}.
\end{equation}
On the other hand, by direct calculations we have
\begin{equation}\label{calculation 1}\begin{split}
\Delta_{\omega_j} e(F_j) &= -\frac{(n-1)^2 (T-t_j)^2 u''}{(u')^3} - \frac{(n-1)(T-t_j)^2u'''}{u'' (u')^2}\\ &\le -\frac{(n-1)(T-t_j)^2u'''}{u'' (u')^2}\le C(n), 
\end{split}\end{equation}
where in the last inequality we use Lemma \ref{elementary bound}. Combining  \eqref{Bochner formula}, \eqref{calculation 1} and the Type I condition $|\ric_{\omega_j}|\le C$, we have
\begin{equation*}
|\nabla\nabla F_j|^2\le C(n) + C e(F_j) + C e(F_j)^2\le C(n).
\end{equation*}
 Therefore, the maps $F_j: X_j\to \cp^{n-1}$ have uniform second order estimates. 
\end{proof}
\begin{remark}
The holomorphicity and so the harmonicity of the maps $F_j$ implies $F_j$ satisfy uniform $C^k$ estimates locally for any $k\in \mathbb Z$. But the second order estimate is enough for our applications.
\end{remark}

The target manifold of $F_j$ is the compact $(\cp^{n-1}, \omega_{FS})$, and by Lemmas \ref{C1 estimate} and \ref{C2 estimate}, the maps $F_j: X_j\to \cp^{n-1}$ have uniformly bounded $C^1$, $C^2$ bounds, hence $F_j$ converge in $C^{1,\alpha}$ topology to a limit map $F_\infty: X_\infty \to \cp^{n-1}$, where by definition $X_j\to X_\infty$ in the $C^\infty$-CG sense with the Riemannian metrics and complex structures converging smoothly. Since $F_j$ are holomorphic with the given complex structures, the limit map $F_\infty: X_\infty \to \cp^{n-1}$ is also holomorphic with respect to the limit complex structure $J_\infty$ on $X_\infty$. And the $C^{1,\alpha}$ convergence, \eqref{lower energy bound} and Lemma \ref{C1 estimate}  imply that the differential map $dF_\infty: TX_\infty \to T\cp^{n-1}$ has full rank $n-1$ at any point in $X_\infty$, hence implicit function theorem implies that the fibers of $F_\infty$ are smooth complete Riemann surfaces.  

We remark that the convergence of $F_j\to F_\infty$ is in the Cheeger-Gromov sense, that is, the maps $\phi_j^* F_j$ converge to $F_\infty$ in uniform $C_{loc}^{1,\alpha}$ topology on any compact subset of $X_\infty$, where $\phi_j: U_j\to X_j$ is the diffeomorphism we chose in Section \ref{CG convergence} realizing the  $C^\infty$-Cheeger-Gromov convergence. 


\subsection{Holomorphic vector fields}
Let $V = \sum_i z_i\frac{\partial}{\partial z_i}$ be a holomorphic vector field on $\mathbb C^n\backslash \{0\}$, which extends to a  holomorphic vector field on $X$, and vanishes on the exceptional divisor $E$. Clearly $V$ is tangential to the fibers of $F_j:X_j\to \cp^{n-1}$.

\begin{lemma}
With respect to a K\"ahler metric $\omega = \ddbar u$ with Calabi symmetry, the imaginary part Im$(V)$ of $V$ is a Killing vector field. Moreover, Im$(V)$ is also Killing with respect to the restriction of the metric on each fiber of $F_j: X_j\to \cp^{n-1}$.
\end{lemma}
\begin{proof}
This follows from straightforward calculations.
Observe that 
$Vu = u'$ and
$$L_V \omega = d (\iota_V \ddbar u ) = d\big(-i \bar\partial (u') \big) = -\ddbar u' ,$$ taking conjugate on both sides we have $L_{\bar V} \omega = -\ddbar u'$, hence it holds that $L_{V-\bar V}\omega = 0$ and this implies the imaginary part of $V$, $\mathrm{Im}(V)$, is a Killing vector field with respect to the metric $\omega$, i.e.
\begin{equation}\label{Killing field}
L_{\mathrm{Im}(V)}\omega = 0.
\end{equation}
On the other hand, for any fiber $F_p$ of $F_j: X_j\to \cp^{n-1}$, we denote the restriction of the metric $\omega$ on this fiber by $\omega_|$ and $i: F_p \to  X$ the inclusion map of the fiber in $X$. Using the fact that Im$(V)$ is tangential to $F_p$ and $L_{\mathrm{Im}(V)}\omega = 0$, when pulled back by the map $i$, we have
$$L_{\mathrm{Im}(V)} (\omega_|) = 0,$$ 
hence Im$(V)$ is also a Killing vector field on $(F_p,\omega_|)$. 

\end{proof}
\begin{remark}We note that the equation \eqref{Killing field} only involves the first order derivatives of $V$.\end{remark}

\begin{lemma}\label{lemma on V} For any $R>0$, there exist $c(n,R)>0$ which goes to $\infty$ as $R\to \infty$ and $C(n,R)>0$, such that if $j$ is large enough, then
 
$$\inf_{\partial B_{g_j}(E_j,R) } |V|_{g_j}^2\ge c(n,R),$$ and
$$\sup_{B_{g_j}(E_j, R)}\abs{V}_{g_j}\le C(n,R).$$
\end{lemma}
\begin{proof}
Applying the expansion formula \eqref{expansion of omega} of the metric $g_j$  we have
\begin{equation}\label{norm of V}
\abs{V}_{g_j} = \frac{u''(\rho,t_j)}{T-t_j},
\end{equation}
so by Lemma \ref{estimate lemma}, we have 
\begin{equation}\label{C 0 V} \abs{V}_{g_j}\le C(n,R) = O(R^2),\quad \text{in } B_{g_j}(E_j,R)\end{equation}
and\begin{equation}\label{nontrivial vector field} \text{on }\partial B_{g_j}(E_j,R), \quad \abs{V}_{g_j}\ge c(n,R)\to \infty,~~ \text{as }R\to \infty. \end{equation}
\end{proof}
Lemma \ref{lemma on V} implies in $B_{g_j}(E_j, R)$, $V$ is a nontrivial holomorphic vector field which vanishes exactly at $E_j$ and has uniform positive lower bound on the boundary $\partial B_{g_j}(E_j,R)$.

We will estimate the bound of the derivatives of $V$ with respect to $g_j$. 
\begin{lemma}\label{lemma 100}
There exists a constant $C=C(n)>0$ such that $\abs{\nabla_j V}_{g_j}\le C$ for any $j$, where $\nabla_j V$ denote the covariant derivative of $V$ with respect to the metric $g_j$.
\end{lemma}
\begin{proof}
We will write $V= V^i \frac{\partial}{\partial z_i}$ for $V^i = z_i$, then $\abs{\nabla V} = V^i_{,k}\overline{V^i_{,k}}$,
where $$V_{,k}^i = \frac{\partial}{\partial z_k} V^i + \Gamma_{kl}^iV^l$$ is the covariant derivative of $V$
and $\Gamma^i_{kl} = g^{i\bar p}\frac{\partial}{\partial z_l} g_{k\bar p}$ is the Levi-Civita connection of a K\"ahler metric $g$. Use the expansion formula \eqref{expansion of omega} (multiplied by $(T-t_j)^{-1}$) of the metric $g_j$, we have
$$V^i_{,k} = \frac{u''}{u'} \delta_{ik} + \bk{\frac{u'''}{u''} - \frac{u''}{u'}} z_i \bar z_k e^{-\rho},$$
observe that when restricted to the exceptional divisor $E = (\rho = -\infty)$ the matrix $(V^i_{,k})$ is of the form (hence has rank $1$)
\begin{equation}\label{limit matrix}
\nabla V|_{E} = \mathrm{diag}(1,0,\ldots,0).
\end{equation} 
We calculate the norm of $\nabla V$:
\begin{equation}\label{gradient V}\begin{split}
\abs{\nabla V}_{g_j} & = n\bk{\frac{u''}{u'}}^2 + 2 \frac{u''}{u'}\bk{ \frac{u'''}{u''} - \frac{u''}{u'}} + \bk{\frac{u'''}{u''} - \frac{u''}{u'}}^2\\
& = (n-1)\bk{\frac{u''}{u'}}^2 + \bk{\frac{u'''}{u''}}^2.
\end{split}\end{equation}
Hence Lemma \ref{lemma 100} follows from the estimates in Lemma \ref{elementary bound}.
\end{proof}

So we have uniform $C^1$ bounds of $V$ with respect to $g_j$. Next we would derive the $C^2$ bounds of $V$ with respect to the metrics $g_j$ on any metric balls $B_{g_j}(E_j, R)$. 
\begin{prop}\label{C2 bound of V}
For any $R>0$, there is a constant $C(n,R)>0$ such that for $j$ large enough we have
\begin{equation*}
\sup_{B_{g_j}(E_j, R)}\bk{|\nabla\nabla V|^2_{\omega_j} + |\bar\nabla\nabla V|_{\omega_j}^2} \le C(n,R),
\end{equation*}
i.e., the $C^2$ bounds of $V$ with respect to the K\"ahler metrics $\omega_j$ hold uniformly on any metric ball $B_{g_j}(E_j, R)$.
\end{prop}

To prove Proposition \ref{C2 bound of V}, we need the following Bochner type identity.
\begin{lemma}
We have the Bochner type identity: for a K\"ahler metric $\omega$,
\begin{equation}\label{bochner identity 1}\begin{split}
\Delta_\omega \abs{\nabla V} =& \abs{\nabla\nabla V}_\omega + \abs{\bar \nabla\nabla V}_\omega + R_{l\bar m} V^i_{,m}\overline{V^i_{,l}} - R_{m\bar i} V^m_{,l} \overline{V^i_{,l}}\\
&- 2 Re\bk{ R_{\bar i m \bar k l} V_{,k}^m \overline{V^i_{,l}} + R_{m\bar i, l} V^m \overline{V^i_{,l}} }.
\end{split}\end{equation}
\end{lemma}
\begin{proof}
This is a direct calculation. 
\begin{equation}\label{bochner equation 1}\begin{split}
\Delta \abs{\nabla V} & = (V^i_{,l} \overline{V^i_{,l}})_{k\bar k}\\
& = V^i_{lk\bar k}\overline{V^i_{,l}} + V^i_{,lk}\overline{V^i_{,lk}} + V^i_{,l\bar k} \overline{V^i_{,l\bar k}} + V^i_{,l} \overline{V^i_{,l\bar k k}}.
\end{split}\end{equation}
By changing the indices, we have
\begin{equation}\label{bochner equation 2}\begin{split}
V_{,lk\bar k}^i &= V^i_{,l\bar k k} + V^i_{,m} R_{l\bar m k\bar k}  - V^m_{,l} R_{~m\bar k k}^i\\
&= V^i_{,l\bar k k} + V^i_{,m} R_{l\bar m}  - V^m_{,l} R_{~m}^i,
\end{split}\end{equation}
and
\begin{equation}\label{bochner equation 3}
V^i_{,l\bar k k} = \bk{V^i_{,\bar k l} - V^m R^i_{~m\bar k l}}_k = - V^m_{,k} R^i_{~m\bar k l}  - V^m R^i_{~m,l},
\end{equation}
where we use the fact that $V$ is a holomorphic holomorphic vector field and the second Bianchi identity. Combining the formulas \eqref{bochner equation 1}, \eqref{bochner equation 2} and \eqref{bochner equation 3}, we can see \eqref{bochner identity 1}.
\end{proof}

\begin{lemma}\label{lemma above}
On the balls $B_{g_j}(E_j, R)\subset X_j$, there exists a constant $C(n,R)>0$ such that 
\begin{equation*}
\Delta_{\omega_j} \abs{\nabla V}_{\omega_j}\le C(n,R), \quad \forall j>>1.
\end{equation*}
\end{lemma}
\begin{proof}

From \eqref{gradient and Laplacian} and \eqref{gradient V} we have
\begin{equation}\label{Laplacian V}
\Delta_{\omega_j}\abs{\nabla V} = (n-1)(T-t_j) \frac{(\abs{\nabla V})'}{u'} + \frac{T-t_j}{u''} \bk{\abs{\nabla V}}'',
\end{equation}
where as before $u' = \frac{\partial}{\partial\rho} u(\rho, t_j)$, etc. Our goal is to show that both terms on RHS of \eqref{Laplacian V} are  uniformly bounded on the balls $B_{g_j}(E_j,R)$. To begin with, we need to estimate $u^{(4)}$.

\noindent{\bf Claim:}
There is a uniform constant $C=C(n)>0$ such that
\begin{equation*}
|u^{(4)}|\le C \frac{(u'')^2}{T-t} + \frac{(u''')^2}{u''}.
\end{equation*}
\begin{proof}[Proof of the {\bf Claim}]
By the formula of scalar curvature (see \cite{Song1}), we have
\begin{equation}\label{scalar curvature}
R(\omega(t)) = - \frac{u^{(4)}}{(u'')^2} + \frac{(u''')^2}{(u'')^3} - 2(n-1) \frac{u'''}{u'u''} - (n-1)(n-2)\frac{u''}{(u')^2} + \frac{n(n-1)}{u'}.
\end{equation}
And by Type I condition we have $|R|\le \frac{C}{T-t}$. Combining with Lemma \ref{elementary bound}, it is easy to see the bound on $|u^{(4)}|$.
\end{proof}
The first term on RHS of \eqref{Laplacian V} is equal to
\begin{align*}
&\frac{(n-1)(T-t_j)}{u'}\bk{ 2(n-1) \frac{u''}{u'}\frac{u'''u'- (u'')^2}{(u')^2}  + 2 \frac{u'''}{u''}\frac{u^{(4)}u'' - (u''')^2}{(u'')^2} }\\
=&\frac{2(n-1)^2(T-t_j)}{u'}\frac{u''}{u'} \cdot\frac{u''' u' - (u'')^2}{(u')^2} + \frac{2(n-1)(T-t_j) }{u'}\frac{u'''}{u''}\cdot\frac{u^{(4)} u'' - (u''')^2}{(u'')^2},
\end{align*}
by examining the terms above using Lemma \ref{elementary bound} and {\bf Claim} we see that the first term on RHS of \eqref{Laplacian V} is uniformly bounded above by $C=C(n)>0$.

The second term in RHS of \eqref{Laplacian V} is a little complicated, after some calculations and replacing the $u^{(4)}$ by the scalar curvature \eqref{scalar curvature}, the second term in RHS of \eqref{Laplacian V} is equal to
\begin{equation}\label{equation middle}\begin{split}
&4(n-1)(n-3)\frac{T-t_j}{u''}\bk{ \frac{ \big((u''')^2+ u'' u^{(4)} \big)u' - u''' (u'')^2  }{(u')^3}  } - 6(n-1)\frac{T-t_j}{u''} \bk{\frac{u''}{u'}}^2 \frac{u'''u' - (u'')^2}{(u')^2}\\
&-\frac{T-t_j}{u''} \bk{ 2 R' u'''+ 2 R u^{(4)}  } - 2(n-1) \frac{T-t_j}{u''} \bk{ \frac{ 2 u' u'' u''' u^{(4)} - (u''')^2(u'u''' + (u'')^2)  }{(u' u'')^2} }\\
& + 2n(n-1) \frac{T-t_j}{u''}\frac{u^{(4)} u' - u''' u''}{(u')^2}.
\end{split}\end{equation}
We look at the third term in \eqref{equation middle}. 
By the Type I condition and Shi's derivative estimate along Ricci flow, we know $|\nabla R(\omega(t_j))|\le \frac{C}{(T-t_j)^{3/2}}$, and also we know $\abs{\nabla R} =  \frac{(R')^2}{u''}$, hence 
\begin{equation*}
|R'|\le C\frac{\sqrt{u''}}{(T-t_j)^{3/2}},
\end{equation*}
so we have
\begin{align*}
\Big|-\frac{T-t_j}{u''} \bk{ 2 R' u'''+ 2 R u^{(4)}  }\Big|&\le C \frac{T-t_j}{u''}\bk{ \frac{\sqrt{u''}}{(T-t_j)^{3/2}} |u'''| + \frac{(u'')^2}{(T-t_j)^2} + \frac{1}{T-t_j} \frac{(u''')^2}{u''} }\\
&\le C(n,R),
\end{align*}
by the Lemmas \ref{elementary bound}, \ref{estimate lemma} and the {\bf Claim}.

The other terms in \eqref{equation middle} can be estimated similarly using the lemmas above, and we can see they are all uniformly bounded. Hence we finish the proof of Lemma  \ref{lemma above}.
\end{proof}
\begin{proof}[Proof of Proposition \ref{C2 bound of V}]
Combining with the Bochner identity \eqref{bochner identity 1}, Type I condition and Shi's derivative estimates, i.e., $|Rm(g_j)|_{g_j}, |\nabla Rm(g_j)|_{g_j}\le C$, and Lemma \ref{lemma above}, we can get the bound on $\abs{\nabla\nabla V}_{\omega_j} + \abs{\bar\nabla\nabla V}_{\omega_j}$.

\end{proof}
\begin{prop}\label{limit holomorphic vector field}
There exists a nontrivial holomorphic vector field $V_\infty$ as the subsequential limit of $V$ along the Cheeger-Gromov convergence, such that $V_\infty$ is tangential to the fibers of $F_\infty:X_\infty\to\cp^{n-1}$ and Im$(V_\infty)$ is a nontrivial Killing vector field on each fiber of $F_\infty$.  
\end{prop}
\begin{proof}
Along the Cheeger-Gromov convergence \eqref{smooth convergence}, by the (locally) uniform $C^0,C^1, C^2$ bound of the holomorphic vector fields $V_j = V$ with respect to the metrics $\omega_j$, up to a subsequence, $V_j$ converge in $C_{loc}^{1,\alpha}$ norm (in the Cheeger-Gromov sense) to a vector field $V_\infty$ on $X_\infty$, which is holomorphic with respect to the complex structure $J_\infty$.  The holomorphic vector field $V_\infty$ satisfies similar $C^0, C^1,C^2$ bounds as $V_j$, when restricted on the balls $B_{g_\infty}(p_\infty,R)$. 

To see $V_\infty$ is nontrivial, there exists a sequence of points $x_j\in \partial B_{g_j}(E_j, R)$ converging to an $x_\infty \in X_\infty$, by \eqref{nontrivial vector field}, we see that $|V_\infty|(x_\infty)\ge c(n,R)>0$, hence $V_\infty$ is nontrivial. 

On the other hand, the vector fields $V_j$ vanish identically on the exceptional divisors $E_j$ in $B_{g_j}(E_j,R)$, and by taking limits, $V_\infty$ also has zero points, e.g. $V_\infty(p_\infty) = 0$. Hence the zero set of $V_\infty$ is a nonempty analytic set, since $V_\infty$ is a holomorphic vector field, and we denote this zero set by $\tilde E_\infty$. It's clear that if a sequence of points $x_j\in E_j$ converges to $x_\infty\in X_\infty$, then $x_\infty\in \tilde E_\infty$.  

Since $V_j$ is tangential to the fibers of $F_j: X_j\to \cp^{n-1}$, $dF_j(V_j) = 0$, from the $C^{1,\alpha}$ convergence of $F_j, V_j$, the limit vector field $V_\infty$ satisfies $dF_\infty(V_\infty) = 0$, i.e., $V_\infty$ is  tangential to the fibers of $F_\infty: X_\infty\to \cp^{n-1}$. 

Choose a fiber $F_\infty^{-1}(y)$ of $F_\infty$ (here $y\in \cp^{n-1}$). There exists a sequence of points $x_j\in F_j^{-1}(y)\cap E_j$ which converge up to a subsequence to $x_\infty\in X_\infty\cap F_\infty^{-1}(y)$, such that $V_\infty (x_\infty) = 0$. On the other hand, for any other point $x_\infty'\in F_\infty^{-1}(y)$, we may assume $d_\infty(x_\infty,x_\infty') = R>0$ and there exists a subsequence of $x_j'\in X_j\cap F_j^{-1}(y)$ with $d_{g_j}(x_j,x_j') >R/2>0$ which converges to $x_\infty'$, then by \eqref{nontrivial vector field}, we see $|V_\infty|_{g_\infty}(x_\infty')\ge c(n,R)>0$. We remark that \eqref{gradient V} implies $\abs{\nabla V_\infty}_\infty (x_\infty)= 1$.


Thus on each fiber $F_\infty^{-1}(y)$ of $F_\infty$,  $V_\infty$ is a holomorphic vector field with simple single zero point. From \eqref{Killing field} and $C^{1,\alpha}$ convergence of $V_j$, the imaginary part $\mathrm{Im}(V_\infty)$ of $V_\infty$ is a Killing vector field of $g_\infty$. Since $\mathrm{Im}(V_\infty)$ is tangential to the fiber, it follows that on the fiber $F_\infty^{-1}(y)$, with respect to the restriction metric  of $g_\infty$ to $F_\infty^{-1}(y)$, the vector field $\mathrm{Im}(V_\infty)$ is also Killing. 
\end{proof}

\begin{corr}
The fibers of $F_\infty:X_\infty\to \cp^{n-1}$ are either biholomorphic to $\mathbb C$ or the disk $D\subset\mathbb C$.
\end{corr}
\begin{proof}
Fix any fiber $F_\infty^{-1}(y)$ of $F_\infty$, which is a complete noncompact Riemann surface.  
From the proof of Proposition \ref{limit holomorphic vector field}, we know the vector field $\mathrm{Im}(V_\infty)$ is Killing in $F_\infty^{-1}(y)$ and has a single zero point in $F_\infty^{-1}(y)$, from Lemma 1 in \cite{CLN}, we  conclude that topologically $F_\infty^{-1}(y)$ is $\mathbb R^2$, which in particular is simply connected. By the uniformization theorem for Riemann surfaces, $F_\infty^{-1}(y)$ is either $\mathbb C$ or the holomorphic disk $D\subset \mathbb C$.
\end{proof}
\section{$U(n)$-actions on the limit space $X_\infty$}\label{section 4}

\subsection{$U(n)$-actions}

We first define a metric on the compact Lie group $U(n)$ by 
\begin{equation}\label{metric on Un}d_0(\sigma_1,\sigma_2):=\max\{d_{\mathbb C^n}(\sigma_1(x),\sigma_2(x))| \text{ for all }x\in S^{2n-1}\subset \mathbb C^n\}\end{equation}
where $d_{\mathbb C^n}$ is the Euclidean distance on $\mathbb C^n$ and $\sigma_1,\sigma_2\in U(n)$ act in the standard way on $S^{2n-1}\subset \mathbb C^n$. We remark that the metrics on the compact group  $U(n)$ are all equivalent, so any other metrics on $U(n)$ will play the same role.
\begin{lemma}
 $d_0$ defines a metric on the compact group $U(n)$.
\end{lemma} 
\begin{proof}We only need to prove that $d_0$ satisfies the triangle inequality, since the $U(n)$-action on $S^{2n-1}$ is effective. For any $\sigma_1,\sigma_2,\sigma_3\in U(n)$, any $\epsilon>0$, there exists an $x_\epsilon\in S^{2n-1}$ such that $d_0(\sigma_1,\sigma_2)\le d_{\mathbb C^n}(\sigma_1(x_\epsilon),\sigma_2(x_\epsilon)) + \epsilon$, then
$$d_0(\sigma_1,\sigma_2)\le d_{\mathbb C^n}(\sigma_1(x_\epsilon),\sigma_3(x_\epsilon)) + d_{\mathbb C^n}(\sigma_2(x_\epsilon),\sigma_3(x_\epsilon))+\epsilon\le d_0(\sigma_1,\sigma_3) + d_0(\sigma_2,\sigma_3) + \epsilon,$$
then letting $\epsilon\to 0$ we can get the triangle inequality.
\end{proof}

For each $\sigma\in U(n)$, we consider the map $\chi_{j,\sigma}$  
\begin{equation}\label{eqn:chi j}\chi_{j,\sigma}: (X_j, g_j,J_j)\to (X_j,g_j,J_j),\end{equation}
defined by $\chi_{j,\sigma}(x) = \sigma(x)$. Recall the $U(n)$-action on $X_j=X$ is induced from the standard $U(n)$-action on $\mathbb C^n\backslash\{0\}$. $\sigma$ acts isometrically  and holomorphically on $(X_j,g_j,J_j)$, so  $\chi_{j,\sigma}$ is a holomorphic isometry. Thus the energy density of $\chi_{j,\sigma}$, $|\nabla_j \chi_{j,\sigma}|^2_{g_j} = n$, where $\nabla_j$ is the connection induced from $ g_j$ and $\chi_{j,\sigma}^*  g_j$. Since $\chi_{j,\sigma}$ is holomorphic, hence also  harmonic. For notation convenience we denote $F = \chi_{j,\sigma}$, then by Bochner formula,
\begin{equation}\label{2 bochner formula}
0=\Delta_j \abs{\nabla_j F} = \abs{\nabla\nabla_j F} + \ric_{ g_j}(\nabla_j F, \overline{\nabla_j F}) - R(F^* g_j)_{\bar \alpha\beta \bar \gamma\delta} \overline{F^\alpha_{,i}} F^\delta_{,i} F^\beta_{,k}\overline{F^\gamma_{,k}},
\end{equation}
where $\Delta_j = \Delta_{ g_j}$ and  $R(F^*g_j)_{\bar\alpha\beta\bar \gamma\delta}$ denotes the sectional curvature of the pulled-back metric $F^* g_j$, which is uniformly bounded by the Type I condition, so is the Ricci curvature of $g_j$. Hence  by \eqref{2 bochner formula} and  $\abs{\nabla_j F} = n$, we see that $\abs{\nabla\nabla_j F}\le C$ for a uniform constant $C=C(n)$. Therefore, we get the  uniform $C^2$ bound of the maps $\chi_{j,\sigma}$, independent of $j,\sigma$. 

Since $\chi_{j,\sigma}$ is an isometry and maps $E_j$ to itself, which has fixed diameter $D_n$ under the metric $g_j$, we have for any $R>0$, the image of $B_{g_j}(p_j,R)$ under $\chi_{j,\sigma}$ is contained in $B_{g_j}(p_j, R+D_n)$. Therefore, the maps $\chi_{j,\sigma}$ are locally uniformly bounded, and satisfy uniform $C^1,C^2$ bounds, so along the Cheeger-Gromov convergence \eqref{smooth convergence},  up to a subsequence of $j$, $\chi_{j,\sigma}$ converge to a limit map 
\begin{equation*}
\chi_{\infty,\sigma}:(X_\infty,g_\infty,J_\infty)\to (X_\infty,g_\infty,J_\infty),
\end{equation*} 
which  preserves the metric $g_\infty$ and complex structure $J_\infty$, hence an isometry and holomorphic map. The map $\chi_{\infty,\sigma}$ is defined through a subsequence of $\chi_{j,\sigma}$. For different $\sigma\in U(n)$, the subsequence might be different. Our next lemma will show that there exists a subsequence of $j$, such that for {\em all} $\sigma\in U(n)$, $\chi_{j,\sigma}$ converge to limit maps $\chi_{\infty,\sigma}$.

\begin{lemma}\label{lemma 4.2}
For any $R>0$, there exists a $C(n,R)>0$ such that for $j$ large enough, we have
\begin{equation*}
d_{g_j}(\sigma_1(x),\sigma_2(x))\le C(n,R)d_0(\sigma_1,\sigma_2), \quad \forall \sigma_1,\sigma_2\in U(n)
\end{equation*}
and $x\in B_{g_j}(E_j,R)\subset (X_j,g_j,J_j)$, where $d_0$ is the metric on $U(n)$ defined in \eqref{metric on Un}
\end{lemma}
\begin{proof}
By the expansion formula of $g_j = \frac{1}{T-t_j}g(t_j)$ in \eqref{expansion of omega}, and Lemma \ref{estimate lemma} we have on $B_{g_j}(E_j,R)\backslash E_j\subset \mathbb C^n\backslash\{0\}$  (here we identify $B_{g_j}(E_j,R)\backslash E_j$ as a punctured ball in $\mathbb C^n\backslash \{0\}$)
\begin{equation}\label{gj local}\begin{split}g_j &= \frac{u'}{T-t_j}\bk{ \frac{\delta_{ik}}{\abs{z}} - \frac{\bar z_i z_k}{|z|^4} }dz_i\wedge d\bar z_k + \frac{u''}{T-t_j} \frac{\bar z_i z_k}{|z|^4} dz_i\wedge d\bar z_k\\
&\le \frac{C(n,R)}{\abs{z}} \omega_{\mathbb C^n},
\end{split}\end{equation}
and $\omega_{\mathbb C^n}$ is the Euclidean metric on $\mathbb C^n$, so for any $x\in B_{g_j}(E_j,R)\backslash E_j\subset \mathbb C^n\backslash\{0\}$, and $\sigma_1,\sigma_2\in U(n)$, $\sigma_1(x), \sigma_2(x)$ remain in $B_{g_j}(E_j,R)\backslash E_j\subset \mathbb C^n\backslash\{0\}$ and the Euclidean norm $|\sigma_1(x)| = |\sigma_2(x) |= |x|$. Choose a curve $\gamma\subset S^{2n-1}_{|x|}$, the Euclidean sphere in $\mathbb C^n\backslash\{0\}$ with radius $|x|$, connecting $\sigma_1(x)$ and $\sigma_2(x)$ and the Euclidean length $L_{\mathbb C^n}(\gamma)\le 2d_{\mathbb C^n}(\sigma_1(x),\sigma_2(x))$. Hence by the estimate \eqref{gj local}, we have 
\begin{equation}\label{Un}\begin{split}
d_{g_j}(\sigma_1(x), \sigma_2(x))&\le d_{g_j}(\gamma)\\ &\le \frac{C(n,R)}{|x|} L_{\mathbb C^n}(\gamma) \\
&\le \frac{2C(n,R)}{|x|} d_{\mathbb C^n}(\sigma_1(x),\sigma_2(x)) \\
&= 2C(n,R)d_{\mathbb C^n}\bk{\sigma_1\Big(\frac{x}{|x|}\Big), \sigma_2\Big(\frac{x}{|x|}\Big)}\\
&\le 2C(n,R) d_0(\sigma_1,\sigma_2).\end{split}\end{equation}
By continuity, \eqref{Un} also holds for $x\in E_j$.

\end{proof}
If we define maps
\begin{equation}\label{maps}\chi_j: (X_j, g_j,J_j)\times (U(n), d_0)\to (X_j,g_j,J_j)\end{equation}
by $\chi_j(x,\sigma) = \chi_{j,\sigma}(x)$, which are holomorphic in $x$ and  satisfy
$$d_{ g_j}(\chi_j(x,\sigma), \chi_j(y,\sigma)) = d_{g_j}(x,y),\quad \forall x, y\in X_j,\sigma\in U(n),$$
and by Lemma \ref{lemma 4.2} we also have $$d_{g_j}(\chi_j(x,\sigma_1),\chi_j(x,\sigma_2))\le C(n,R)d_0(\sigma_1,\sigma_2),\quad \forall x\in B_{g_j}(E_j,R), \sigma_1,\sigma_2\in U(n).$$
Hence for any $x,y\in B_{g_j}(E_j,R)$ and $\sigma_1,\sigma_2\in U(n)$
\begin{equation}\begin{split}
d_{g_j}(\chi_j(x,\sigma_1), \chi_j(y, \sigma_2))& \le d_{g_j}(\chi_j(x,\sigma_1),\chi_j(y,\sigma_1)) + d_{g_j}(\chi_j(y,\sigma_1),\chi_j(y,\sigma_2))\\
&\le d_{g_j}(x,y) + C(n,R)d_0(\sigma_1,\sigma_2)
\end{split}\end{equation}
which implies the maps $\chi_j$ defined in \eqref{maps} are locally uniformly bounded and locally equi-continuous with respect to the given product metrics.  Moreover the maps $\chi_j(\cdot,\sigma)$ satisfy uniform $C^1, C^2$ bounds for any $\sigma\in U(n)$, hence by Arzela-Ascoli theorem, up to a subsequence of $j$, $\chi_j$ converge to a map
\begin{equation}\label{limit map}\chi_\infty: (X_\infty,g_\infty,J_\infty)\times (U(n),d_0)\to (X_\infty, g_\infty, J_\infty),\end{equation}
and for each $\sigma\in U(n)$, the map $$\chi_\infty(\cdot,\sigma): (X_\infty, g_\infty, J_\infty) \to (X_\infty, g_\infty,J_\infty)$$ is an isometry and $J_\infty$-holomorphic. 


\begin{lemma}\label{product lemma}
The map $\chi_\infty$ defined in \eqref{limit map} satisfies
\begin{equation}\label{group structure}
\chi_\infty(x,\sigma_1\sigma_2) = \chi_\infty(\chi_\infty(x,\sigma_2), \sigma_1), \quad \forall x\in X_\infty, \sigma_1,\sigma_2\in U(n).
\end{equation}
\end{lemma}
\begin{proof}
For any $x\in X_\infty$ and $\sigma_1,\sigma_2\in U(n)$, choose a sequence of $x_j\in X_j$ converging to $x$. For each $j$ from the definition we have
\begin{align*}
\chi_j(x_j, \sigma_1\sigma_2) &= \chi_{j,\sigma_1\sigma_2} (x_j) = \sigma_1\sigma_2(x_j)=\sigma_1\big(\sigma_2(x) \big)\\
& = \chi_j(\sigma_2(x_j), \sigma_1) = \chi_j\big(\chi_j(x_j, \sigma_2),\sigma_1   \big),
\end{align*}
taking $j\to \infty$ and by the definition of $\chi_\infty$ we have
$$\chi_\infty(x,\sigma_1\sigma_2) = \chi_\infty\big(\chi_\infty(x,\sigma_2),\sigma_1 \big).$$
\end{proof}

\begin{remark}
If we define the ``action'' of $\sigma\in U(n)$ on $X_\infty$, $\sigma: X_\infty\to X_\infty$ by $\sigma\cdot x = \chi_\infty(x,\sigma)$, then Lemma \ref{product lemma} means that for any $\sigma_1,\sigma_2\in U(n)$, $(\sigma_1\sigma_2)\cdot x = \sigma_1\cdot(\sigma_2\cdot x)$, for any $x\in X_\infty$.

\end{remark}

It is clear that the identity element $e\in U(n)$ satisfies $\chi_\infty(x, e) = x$, i.e., $e\cdot x = x$ for any $x\in X_\infty$. Hence the $U(n)$-action on $X_\infty$ defined above is a group action.

\subsection{$U(n)$-action and fiber map $F_\infty$}Recall in Section \ref{subsection 3.2}, we define a holomorphic map $F_\infty: X_\infty \to \cp^{n-1}$, as the limit map of $F_j: X_j\to \cp^{n-1}$. It is clear that $F_j$ is $U(n)$-equivariant with respect to the $U(n)$-action on $X_j = \cp^n\# \overline{\cp^n}$ and the standard action on $\cp^{n-1}$, i.e. 
$$F_j(\sigma\cdot x_j) = \sigma\cdot F_j(x_j),\quad \forall x_j\in X_j~~\forall \sigma\in U(n).$$
Now for any $x\in X_\infty$, there is a sequence $x_j\in X_j$ converging to $x$, taking $j\to \infty$ and by the smooth convergence of $F_j$ to $F_\infty$, we have
$$F_\infty(\sigma\cdot x) = \sigma\cdot F_\infty(x),$$
i.e. $F_\infty$ is $U(n)$-equivariant. Hence for any $y\in \cp^{n-1}$, $\sigma\in U(n)$ maps the fiber $F_\infty^{-1}(y)$ to $F_\infty^{-1}(\sigma\cdot y)$.
\begin{lemma}
The restriction of $\sigma:X_\infty \to X_\infty$ to the fiber $F_\infty^{-1}(y)$
$$\sigma|_{F_\infty^{-1}(y)}: F_\infty^{-1}(y) \to F_\infty^{-1}(\sigma\cdot y)$$ is a biholomorphic map.
\end{lemma}
\begin{proof}
This follows from the fact that $$\sigma\sigma^{-1}  = e =id: F_\infty^{-1}(\sigma\cdot y)\to F_\infty^{-1}(\sigma\cdot y),  $$ 
and $$\sigma^{-1}\sigma = e = id : F_\infty^{-1}(y)\to F_\infty^{-1}(y).$$
And both $\sigma$ and $\sigma^{-1}$ are holomorphic maps.
\end{proof}

\begin{corr}\label{corollary 4.1}
The fibers of $f_\infty: X_\infty \to \cp^{n-1}$ are all biholomorphic to each other.
\end{corr}
This follows from the previous lemma and the fact that $U(n)$ action on $\cp^{n-1}$ is transitive.

Fix  $p=[1:0:\cdots:0]\in\mathbb {CP}^{n-1}$, and denote the fiber $F_\infty^{-1}(p)$ by $ F_p$. We know from Corollary \ref{corollary 4.1} all fibers of $F_\infty$ are isomorphic. It is expected that $F_\infty$ is in fact a fiber bundle over $\cp^{n-1}$ with fiber $F_p$.

\begin{prop}\label{fiber prop}
The map $F_\infty: X_\infty \to \mathbb{CP}^{n-1}$ is a fiber bundle with fibers isomorphic to $F_p$.
\end{prop}
\begin{proof}
The compact group $SU(n)$-action on $X_\infty$ induces an action of the complexified group $SL(n,\mathbb C)$ of $SU(n)$, which is defined through the infinitesimal action: for any $\xi + \sqrt{-1} \eta\in \mathfrak{su}(n)\oplus \sqrt{-1} \mathfrak{su}(n) = \mathfrak{sl}(n,\mathbb C)$, we define $\exp(\xi + \sqrt{-1}\eta)\cdot x = \exp(\xi)\exp(J_\infty\eta)\cdot x$, where $J_\infty$ is the complex structure on $X_\infty$.

Define a map $$\pi: SL(n,\mathbb C)\times F_p\to X_\infty,\quad (\sigma, x)\mapsto \sigma\cdot x.$$
This is indeed a {\em surjective} map by the property of group actions. If $\pi(\sigma_1,x_1) = \pi(\sigma_2,x_2)$ for some $\sigma_1,\sigma_2\in SL(n,\mathbb C)$ and $x_1,x_2\in F_p$. Then $\sigma_1\cdot x_1 = \sigma_2 \cdot x_2$,  and $\sigma_1\cdot p = \sigma_1\cdot F_\infty(x_1)=F_\infty(\sigma_1\cdot x_1) = F_\infty(\sigma_2\cdot x_2) = \sigma_2 \cdot p$, therefore, $\sigma_1^{-1}\circ \sigma_2 \in $ isotropic subgroup $B$ of $SL(n,\mathbb C)$ acting on $\mathbb {CP}^{n-1}$, which is given by the matrices of the form
\begin{equation*}
\sigma_1^{-1}\circ\sigma_2 =   \begin{pmatrix}
a& \mathbb{*}\\
\mathbf{0} & A
\end{pmatrix}
\end{equation*}
where $a\in\mathbb C^*$ and $A\in GL(n-1,\mathbb C)$ such that $a\det A = 1$ and $\mathbb{*}$ denotes a vector in $\mathbb C^{n-1}$. Hence we have $x_1 =\begin{pmatrix}
a& \mathbb{*}\\
\mathbf{0} & A
\end{pmatrix}\cdot x_2 $.

We define an equivalence relation on $SL(n,\mathbb C)\times F_p$ as  $$(\sigma_1,x_1)\sim (\sigma_2,x_2)$$ if there exists a matrix $\begin{pmatrix}
a& \mathbb{*}\\
\mathbf{0} & A
\end{pmatrix}\in B$ such that $\sigma_2 = \sigma_1\circ \begin{pmatrix}
a& \mathbb{*}\\
\mathbf{0} & A
\end{pmatrix}$ and $x_2 = \begin{pmatrix}
a& \mathbb{*}\\
\mathbf{0} & A
\end{pmatrix}^{-1}\cdot x_1$. Then we can see that if $(\sigma_1,x_1)\sim (\sigma_2,x_2)$, then $\pi(\sigma_1,x_1) = \pi(\sigma_2,x_2)$. Hence the quotient map
$$\bar \pi: SL(n,\mathbb C)\times F_p/ _\sim\to X_\infty$$ is bijective and also a biholomorphic map, since each action $\sigma\in SL(n,\mathbb C)$ on $X_\infty$ is holomorphic and $SL(n,\mathbb C)$ is a complex manifold.  


\noindent {\bf Claim}: $SL(n,\mathbb C)\times F_p/_\sim$ is a fiber bundle over $\cp^{n-1}$ with fibers isomorphic to $F_p$.

\noindent{\em Proof of the {\bf Claim}:} Define the projection map $pr: SL(n,\mathbb C)\times F_p/_\sim\to SL(n,\mathbb C)/B\cong \cp^{n-1}$, by $pr(\sigma, x) = Q(\sigma)$, where $Q: SL(n,\mathbb C)\to SL(n,\mathbb C)/B$ is the quotient map. $pr$ is clearly well-defined and we want to show $pr$ is locally trivial. The principal $B$-bundle $Q$ is locally trivial, so around any point in $\cp^{n-1}\cong SL(n,\mathbb C)/B$, there is an open set $U$ such that $Q^{-1}(U)\cong U\times B$, i.e. there is a local trivialization $\varphi: Q^{-1}(U)\to U\times B$, and we denote $\varphi = (\varphi_1,\varphi_2)$. By the definition of quotient map $Q$, it is clear that $\varphi_1(\sigma\mathbf{b}) = \varphi_1(\sigma)$ for any $\sigma\in Q^{-1}(U)$ and $\mathbf{b}\in B$. Thus we can define a local section $s:U\to Q^{-1}(U)$ of $Q$ by $s(y) = \varphi^{-1}(y,\mathbf{e})$ with $\mathbf{e}\in B$ being the identity matrix.  

Define a map $\tilde \varphi: pr^{-1}(U) = Q^{-1}(U)\times F_p/_\sim\to U\times F_p$ by $$\tilde\varphi (\sigma,x) = (\varphi_1(\sigma), s(\varphi_1(\sigma))^{-1}\cdot \sigma\cdot x)$$ which by the property of $\varphi_1$ is clearly well-defined. We want to show $\tilde \varphi$ is bijective. $\tilde \varphi$ is clearly surjective. To see that it is also injective, suppose $\tilde\varphi(\sigma_1,x_1) = \tilde\varphi(\sigma_2,x_2)$, then $\varphi_1(\sigma_1) = \varphi_1(\sigma_2)$, so there exists a matrix $\mathbf{b}\in B$ such that $\sigma_2 = \sigma_1 \mathbf{b}$. Since $s(\varphi_1(\sigma_1))^{-1}: F_{\sigma_1\cdot p}\to F_p$ is an isomorphism, we must have $\sigma_1\cdot x_1 = \sigma_2\cdot x_2$, and this implies $x_2 = \mathbf{b}^{-1}\cdot x_1$, and hence $(\sigma_1,x_1)\sim (\sigma_2,x_2)$, and the map $\tilde \varphi$ is injective. In the definition of $\tilde \varphi$, all maps are holomorphic hence $\tilde \varphi$ is also  holomorphic, and $\tilde\varphi$ provides the local trivialization of $SL(n,\mathbb C)\times F_p/_\sim$ over $U\subset \cp^{n-1}$.
\end{proof}

For the fixed point $p=[1:0:\cdots:0]\in \cp^{n-1}$, it is well-known that the isotropic subgroup $U_p$ at $p$ of the $U(n)$-action on $\cp^{n-1}$ is isomorphic to $U(1)\times U(n-1)$ and  given by the the matrices of the form
$$\begin{pmatrix} e^{i\theta} & 0 \\ 0 & A\end{pmatrix},\quad \text{for some }A\in U(n-1),~~e^{i\theta}\in U(1).$$
Each $\sigma\in U_p$ induces an isomorphism of the fiber $F_p$, which is either $\mathbb C$ or the unit disk $D\subset \mathbb C$. 
\begin{lemma}\label{effective lemma}
There exists an $x_0\in F_p$ such that for any $\sigma\in U_p$, $\sigma\cdot x_0 = x_0$. Moreover, if $\sigma\in U_p$ fixes all $x\in F_p$, then $\sigma\in \{1\}\times U(n-1)$, i.e., $\sigma$ is of the form
$$\begin{pmatrix}1 & 0 \\ 0 & A  \end{pmatrix}, \quad \text{for some }A\in U(n-1).$$

\end{lemma}
\begin{proof}
It is clear that $\sigma\in U_p$ also induces an isomorphism of the fibers $F_j^{-1}(p)$. For each $j$, there exists an $x_{0,j} \in F_j^{-1}(p)\cap E_j$ which is fixed by all $\sigma\in U_p$.  Therefore we can assume that $x_{0,j}$ converges to $x_0$ as $j\to\infty$. We then have that   $x_0\in F_\infty^{-1}(p) = F_p$ is the fixed point of all $\sigma\in U_p$.

Suppose there exists a $\sigma\in U_p$ such that $\sigma\cdot x = x$ for all $x\in F_p$. Fix a large $R>0$ and  for any $x_j\in F_j^{-1}(p)\cap B_{g_j}(E_j,R)$ with $x_j\to x_\infty\in F_p$, $d_{g_j}(\sigma\cdot x_j,x_j)\le \epsilon_j\to 0$ as $j\to \infty$, since $d_{g_j}(\sigma\cdot x_j,x_j)\to d_{g_\infty}(\sigma\cdot x_\infty, x_\infty) = 0$. On $\partial B_{g_j}(E_j,R)$, by Lemma \ref{estimate lemma} and the expansion formula of $g_j$ in \eqref{gj local} there exists a constant $c(n,R)>0$ such that
\begin{equation}\label{local estimate 10}g_j\ge c(n,R)\frac{\omega_{\mathbb C^n}}{\abs{z}}.\end{equation}
For any $x_j\in F_j^{-1}(p)\cap \partial B_{g_j}(E_j,R)$, the minimal geodesic $\gamma_j$ (with respect to $g_j$) connecting $x_j$ and $\sigma\cdot x_j$ must be contained in the annulus $B_{g_j}(E_j, R+\epsilon_j)\backslash B_{g_j}(E_j, R -\epsilon_j)\subset \mathbb C^n\backslash \{0\}$, where the estimate \eqref{local estimate 10} still holds with some different $c(n,R)>0$, therefore we have
\begin{equation}\label{eqn:effective}\begin{split}
\epsilon_j\ge d_{g_j}(\sigma\cdot x_j,x_j) & = L_{g_j}(\gamma_j)\\
&\ge c(n,R) d_{\frac{\omega_{\mathbb C^n}}{\abs{z}}}(\gamma_j)\\
&\ge c(n,R)d_{g_{S^{2n-1}}} \bk{ \sigma(\frac{x_j}{|x_j|}), \frac{x_j}{|x_j|} },
\end{split}\end{equation} 
where $g_{S^{2n-1}}$ is the standard metric on the  unit sphere $S^{2n-1}\subset\mathbb C^n\cong\mathbb R^{2n}$ and we use the fact that the metric
\begin{equation*}
\frac{\omega_{\mathbb C^n}}{\abs{z}} = (d\log |z|)^2 + g_{S^{2n-1}},
\end{equation*}
is a product metric on $\mathbb C^n\backslash\{0\}$, so the distance of $x_j$ and $\sigma\cdot x_j\in \mathbb C^n\backslash \{0\}$ with respect to $\frac{\omega_{\mathbb C^n}}{\abs{z}}$ is equal to $d_{g_{S^{2n-1}}}\bk{ \sigma(\frac{x_j}{|x_j|}), \frac{x_j}{|x_j|} }$, since the Euclidean norms $|x_j| = |\sigma\cdot  x_j|$. Suppose $\sigma\in U_p\subset U(n)$ is given by $\begin{pmatrix} e^{i\theta} & 0 \\ 0& A\end{pmatrix}$ for some $A\in U(n-1)$, and it acts on the big circle $F_j^{-1}(p)\cap S^{2n-1}$ by rotation by angle $\theta$. Then \eqref{eqn:effective} means that for any  $x\in F_j^{-1}(p)\cap S^{2n-1}$, $d_{g_{S^{2n-1}}}(\sigma\cdot x, x)$ is arbitrarily small, hence equals to zero, so the rotation angle $\theta = 0$, and $\sigma\in U_p$ is of the form $\begin{pmatrix} 1 & 0 \\ 0& A\end{pmatrix}$ for some $A\in U(n-1)$.
\end{proof}

\begin{remark}
Noting that the automorphism groups of $D$ and $\mathbb C$ are given by
$$\mathrm{Aut}(D) = \Big\{f_{a,\theta}| f_{a,\theta}(\zeta) = e^{i\theta}\frac{\zeta-a}{1-\bar a \zeta},\quad \theta\in S^1,~a\in D\Big\},$$
$$\mathrm{Aut}(\mathbb C) = \Big\{a\zeta + b| a,b\in\mathbb C,~a\neq 0\Big\},$$
respectively. The action of each nonidentity $\sigma\in U_p$ on $F_p$ is of one of the above, hence has one and only one fixed point  in $F_p$.
\end{remark}
We know holomorphic line bundles over $\cp^{n-1}$ are given by $\mathcal O_{\cp^{n-1}}(k)$ for some $k\in\mathbb Z$. And each  fiber $F_p$ can be embedded in the complex line $\mathbb C$ with the fixed point $x_0$ identified as $0\in\mathbb C$ hence the fiber bundle $F_\infty: X_\infty\to \cp^{n-1}$ can be embedded into some line bundle $\mathcal O_{\cp^{n-1}}(k)$, so that  $X_\infty$ is either the line bundle $\mathcal O_{\cp^{n-1}}(k)$ or the disk bundle as a portion of $\mathcal O_{\cp^{n-1}}(k)$.
\begin{lemma}
We have $k = -1$.
\end{lemma} 
\begin{proof}
We have known from Theorem \ref{Song theorem} (1) (see also \cite{Song1}) that $X_\infty$ is {\em diffeomorphic} to $\widetilde{ \mathbb C^n}$, $\mathbb C^n$ blown-up at one point, so $k$ must be negative and odd. On the other hand, if $k\neq -1$, then the $U_p$ actions on the fiber of $\mathcal O_{\cp^{n-1}}(k)$ over $p\in \cp^{n-1}$ are not ``effective'' in the sense that a matrix of the form $\begin{pmatrix} e^{2\pi i/k} & 0 \\ 0 & A  \end{pmatrix}$ inducing the identity action on the fiber of $\mathcal O_{\cp^{n-1}}(k)$ over $p\in \cp^{n-1}$, and inducing the identity action on $F_p$. This contradicts Lemma \ref{effective lemma}.
\end{proof}

\begin{corr}\label{main cor}
$X_\infty$ is either the holomorphic line bundle $\mathcal O_{\cp^{n-1}}(-1)$ or the holomorphic disk bundle as a portion of $\mathcal O_{\cp^{n-1}}(-1)$.
\end{corr}

\section{Proof of Theorem \ref{main theorem}}\label{section 5}


We first show that the limit metric $g_\infty$ on $X_\infty\subset \mathcal O_{\cp^{n-1}}(-1)$ is $U(n)$ invariant with respect to the natural coordinates of $\mathbb C^n\backslash\{0\}= \mathcal O_{\cp^{n-1}}(-1)\backslash E_\infty$.

\begin{lemma}
There exists a smooth function $U_\infty$ on $X_\infty$, such that 
$$g_\infty = (n-1)F_\infty^*\omega_{FS} + \ddbar U_\infty,$$ where $F_\infty:X_\infty\to \cp^{n-1}$ is the map constructed in Section \ref{section 3}, and $\omega_{FS}$ is the Fubini-Study metric on $\cp^{n-1}$.
\end{lemma}
\begin{proof}
Let $R>0$ be large number. On $B_{g_j}(E_j,R)$ the metrics \begin{equation}\label{gj expansion}g_j = \ddbar\bk{ \frac{u(t_j, \rho)}{T-t_j}} = (n-1)F_j^*\omega_{FS} + \ddbar \bk{\frac{u(t_j,\rho)}{T-t_j} - (n-1)\rho}.\end{equation}
 By the Calabi symmetry condition,  the K\"ahler potentials $u(t,\rho) = (n-1)(T-t)\rho + U_{0}(t,e^\rho)$ near $\rho = -\infty$, and we can normalize for each $t\in [0,T)$,  $U_{0}(t,0) = 0$, hence the smooth functions $\bk{\frac{u(t_j,\rho)}{T-t_j} - (n-1)\rho}|_{E_j}=\bk{\frac{u(t_j,\rho)}{T-t_j} - (n-1)\rho}|_{\rho=-\infty} = 0$ for any $j\ge 1$. Set $U_j = \frac{u(t_j,\rho)}{T-t_j} - (n-1)\rho$. The gradient of $U_j$ with respect to $g_j$ is 
\begin{align*}\abs{\nabla U_j}_{g_j} &= (T-t_j)\frac{(U_j')^2}{u''(t_j,\rho)}\\
&=(T-t_j)\frac{\bk{ \frac{u'(\rho, t_j)}{T-t_j} - (n-1) }^2}{u''(t_j,\rho)}\\
&\le \frac{(u'(t_j,\rho) - (n-1)(T-t_j))}{T-t_j} \frac{u'(\rho,t_j) - a_{t_j}}{u''(t_j,\rho)}\\
&\le C(n,R)\quad \text{ on } B_{g_j}(E_j, R)
\end{align*} for $j$ large enough, 
where in the last inequality we use Lemmas \ref{elementary bound} and \ref{estimate lemma}. Hence $\|U_j\|_{C^0(B_{g_j}(E_j, R))}\le C(n,R)$ for some $C(n,R)>0$. Moreover, the Laplacian of $U_j$ $$\Delta_{g_j} U_j = n - (n-1) tr_{\omega_j} f_j^*\omega_{FS}= n - (n-1)\frac{(n-1)(T-t_j)}{u'(t_j,\rho)}$$
satisfies $\Delta_{g_j} U_j|_{E_j} = 1$ and $$|\nabla \Delta_j U_j|_{g_j}^2 = C(n)\frac{(T-t_j)^3 u''(t_j)}{(u'(t_j))^4}\le C(n),$$
so $$\|\Delta_j U_j\|_{C^1(g_j, B_{g_j}(E_j,R))}\le C(n,R).$$ Hence by elliptic estimate $$\|U_j\|_{C^{2,\alpha}(g_j, B_{g_j}(E_j, R/2))}\le C(n,R).$$ Therefore the functions $U_j$ are locally uniformly bounded in $C^{2,\alpha}$ norm on any compact subset $B_{g_j}(E_j, R)$ of $X_j$. Taking a subsequence and using a diagonal argument, $U_j$ converge (in the Cheeger-Gromov sense) locally uniformly in $C^{2,\alpha}$ topology to some $C^{2,\alpha}$ function $U_\infty$ on $X_\infty$,  
therefore from \eqref{gj expansion}, $C^{1,\alpha}_{loc}$ convergence of the holomorphic maps $F_j$ to $F_\infty$ and smooth convergence of complex structures, the metrics $g_j$ converge in $C^\alpha$ norm to 
\begin{equation}\label{g_infty}g_\infty = (n-1)F_\infty^* \omega_{FS} + \ddbar U_\infty.\end{equation} Since $g_\infty$ and $F_\infty^*\omega_{FS}$ are both smooth, $U_\infty$ is also a smooth function on $X_\infty$. 
\end{proof}

Take coordinates of $\mathcal O_{\cp^{n-1}}(-1)\cong\widetilde{\mathbb C^n}$, $\mathbb C^n$ blown-up at the origin, $\zeta = z_1(\neq 0)$, $w_2 = z_2/z_1$, $\ldots$, $w_n = z_n/z_1$, where $z_1,\ldots,z_n$ are the natural coordinates on $\mathbb C^n$, and $\zeta$ is the  coordinate of fibers and $w_2,\ldots, w_n$ are coordinates of $\cp^{n-1}$. Set $\rho = \log \abs{z} = \log\bk{\abs{\zeta}(1+\abs{w})}$, our goal in this subsection is to show 
\begin{lemma}The function $U_\infty$ constructed in \eqref{g_infty} can be modified to depend only on $\rho$. That is, $U_\infty(\zeta,w) = \tilde U_\infty(\abs{\zeta}(1+\abs{w}))$ for some single-variable function $\tilde U_\infty(\cdot):\mathbb R\to \mathbb R$. Hence $U_\infty$ is $U(n)$ invariant on $\mathbb C^n\backslash \{0\}$. 
\end{lemma}
\begin{proof}
By construction the limit metric $g_\infty$ is invariant under the $U(n)$-action defined in section \ref{section 4}, so we have $$\sigma^* g_\infty = g_\infty,\quad \forall \sigma\in U(n).$$ By \eqref{g_infty} we have $$\sigma^*(\ddbar U_\infty) = \ddbar \sigma^* U_\infty = \ddbar U_\infty,\quad \forall \sigma\in U(n).$$ By averaging the function $U_\infty$ over the compact group $U(n)$ using the Harr measure, we may assume $\sigma^* U_\infty = U_\infty$ for all $\sigma\in U(n)$. Since we identify the unique fixed point of the $U_p$ action in the fiber $F_p$ with the origin in $\mathbb C$, the zero section $E_\infty$ (which is locally given by $\zeta = 0$) of the line bundle $\mathcal O_{\cp^{n-1}}(-1)$ coincide with the fixed point of the $U(n)$-actions in each fiber of $F_\infty:X_\infty\to \cp^{n-1}$.  Since $U(n)$-action is transitive on $\mathbb{CP}^{n-1}$, for any $w=(\zeta,w_2,\ldots w_n)\in X_\infty$, there is some $\sigma_w\in U(n)$ mapping $w$ to $(\zeta',0,\ldots,0)\in X_\infty$ for some $\zeta'\in\mathbb C$ satisfying $\abs{\zeta'} = \abs{\zeta}(1+\abs{w})$.  So (writing $Z=(w_2,\ldots,w_n)$)
 \begin{equation}\label{eqn:mapping}U_\infty(\zeta,\bar\zeta, Z, \bar Z) = U_\infty(\sigma_w\cdot (\zeta,Z) ) = U_\infty(\zeta',\bar \zeta',0,0).\end{equation} 
 
 On the other hand, for $p=(0,\ldots,0)\in \mathbb {CP}^{n-1}$, the isotopy group $U_p\subset U(n)$ at $p$ preserves the fiber $F_\infty^{-1}(p)$, which is either $D\subset \mathbb C$ or $\mathbb C$. The subgroup $U_p$ fixes the point $(\zeta = 0, 0,\ldots, 0)\in E_\infty$, which can be viewed as the origin in the fiber $F_\infty^{-1}(p)$. Noting that the automorphism groups of $D$ and $\mathbb C$ are given by
$$\mathrm{Aut}(D) = \Big\{f_{a,\theta}| f_{a,\theta}(\zeta) = e^{i\theta}\frac{\zeta-a}{1-\bar a \zeta},\quad \theta\in S^1,~a\in D\Big\},$$
$$\mathrm{Aut}(\mathbb C) = \Big\{a\zeta + b| a,b\in\mathbb C,~a\neq 0\Big\},$$
respectively. We see from both cases that the $U_p$ action on the fiber $F_\infty^{-1}(p)$ is given by $\sigma_\theta (\zeta) = e^{i\theta} \zeta$ for $\theta\in S^1$, which means that the $U_p$ action on the fiber is the rotation action of $S^1$ on $\mathbb C$. The property that $U_\infty$ is invariant under the $U_p$ action implies that $$U_\infty(\zeta',\bar\zeta', 0, 0) = U_\infty(|\zeta'|, |\zeta'|, 0 ,0),\quad\forall \zeta'\in\mathbb C,$$ 
combining with \eqref{eqn:mapping}, we see for any $(\zeta,w_2,\ldots,w_n)\in X_\infty$ $$U_\infty(\zeta,w_2,\ldots,w_n)=\tilde U_\infty\big(\abs{\zeta}(1+\abs{w})\big)$$ for some single variable function $\tilde U_\infty$.
\end{proof}

\subsection{Proof of Theorem \ref{main theorem}}
So far we have shown that $X_\infty\subset\mathcal O_{\mathbb{CP}^{n-1}} (-1)$ is a fiber bundle with fibers either disk $D$ or the line $\mathbb C$ and the metric $g_\infty$ is $U(n)$-invariant on $\mathbb C^n\backslash \{0\}\cap (X_\infty\backslash E_\infty)$. We know (Theorem \ref{Song theorem}, or \cite{Song1}) that the metric $g_\infty$ is a {\em complete} gradient K\"ahler Ricci soliton, i.e., for some $f_\infty\in C^{\infty}(X_\infty)$ such that
\begin{equation}\label{eqn:krs}\ric(g_\infty)+ \ddbar f_\infty = g_\infty, \qquad \nabla\nabla f_\infty = 0\end{equation}

Without loss of generality, we can choose $f_\infty$ such that it is invariant under the $U(n)$-action  on $\mathbb C^n\backslash\{0\}$, since both $g_\infty$ and $\ric(g_\infty)$ are invariant under $U(n)$-action. $X_\infty\backslash E_\infty$ can be identified with either a punctured ball $B^*\subset \mathbb C^n\backslash \{0\}$, or $\mathbb C^n\backslash \{0\}$, on which the metric $g_\infty$ can be written as $g_\infty = \ddbar u_\infty$ satisfying the Calabi symmetry condition near $z = 0\in\mathbb C^n$, i.e., $$u_\infty = u_\infty(\rho) = (n-1)\rho + U_0(e^\rho), \quad\text{near }\rho = -\infty$$ 
for some smooth $U_0:(-\epsilon,\epsilon)\to \mathbb R$ such that $U_0'(0)>0$, $U_0''(0)>0$, where $\rho = \log \abs{z}$, and $u_\infty', u_\infty''>0$. The equation \eqref{eqn:krs} is equivalent to the following equation on $X_\infty\backslash E_\infty$,
$$u_\infty^{(4)} - 2\frac{(u_\infty''')^2}{u_\infty''} + n u_\infty''' - (n-1) \frac{(u_\infty'')^3}{(u_\infty')^2} - \big(u_\infty''' u_\infty' - (u_\infty'')^2\big) = 0,$$
where $u_\infty'=\frac{d}{d\rho}u_\infty$. Denote $\phi = u_\infty'$, then by some calculations we see that the above equation is equivalent to 
\begin{equation}\label{reduction equation}
(\log \phi')'+ (n-1)(\log\phi)' - \mu\phi' + \phi - n =0,\quad \text{for some }\mu\in\mathbb R.
\end{equation}
\begin{lemma}
$\mu\neq 0$.
\end{lemma}
\begin{proof}If $\mu = 0$, then for $Q:=\log \det g_\infty + u_\infty = -n\rho + (n-1)\log\phi + \log(\phi') + u_\infty$, we have $Q' = 0$, and this implies the metric $g_\infty$ is KE with $\ric(g_\infty) = g_\infty$. Myers' theorem from Riemannian geometry implies the diameter of $(X_\infty,g_\infty)$ is bounded, however, from previous arguments we know the diameter of $(X_\infty,g_\infty)$ is infinity, hence a contradiction. Thus $\mu\neq 0$.
\end{proof}

As in \cite{FIK}, since $\phi' = u_\infty''>0$, we may write $\phi' = F(\phi)$ for some smooth function $F$ on $\mathbb R^+$, in terms of  which \eqref{reduction equation} can be written as
\begin{equation}\label{F equation}
F' + \bk{\frac{n-1}{\phi} - \mu }F - (n-\phi) = 0,
\end{equation}
and one can solve this first order ODE 
\begin{equation}\label{equation solution}
\phi' = F(\phi) = \frac{\nu e^{\mu\phi}}{\phi^{n-1}} + \frac{\phi}{\mu} - \frac{\mu - 1}{\mu^{n+1}}\sum_{j=0}^{n-1}\frac{n!}{j!}\mu^j \phi^{j+1-n},
\end{equation}
for some constant $\nu\in\mathbb R$. 

At any zero point $\phi_0$ of $F(\phi)$, by  \eqref{F equation}, we know $F(\phi_0)' = n-\phi_0$, and by the intermediate value theorem this implies that $F(\phi)$ has at most two positive zeros $0<a\le b$ satisfying $0<a\le n\le b$. By the Calabi symmetry, we have 
$$\lim_{\rho\to -\infty} \phi(\rho) = n-1,\quad \lim_{\rho\to -\infty}\phi'(\rho) = 0,$$
and $0=\lim_{\rho\to -\infty} \phi' = \lim_{\rho\to-\infty}F(\phi) = F(n-1)$, so $a= n-1$ is a zero of $F$. Plugging $a= n-1$ into \eqref{equation solution} we get
\begin{equation}\label{determine mu}
\frac{\nu e^{\mu a}}{a^{n-1}} + \frac{a}{\mu} - \frac{\mu - 1}{\mu^{n+1}}\sum_{j=0}^{n-1}\frac{n!}{j!}\mu^j a^{j+1-n} = 0
\end{equation}

\begin{prop}
We must have $\mu>0$ and $\nu = 0$.
\end{prop}
\begin{proof}
Suppose $\mu<0$, then for large $\phi>0$, the leading term on the RHS of \eqref{equation solution} is $\phi/\mu$, hence the solution to \eqref{equation solution} exists for all $\rho\in (-\infty,\infty)$, and  $\phi(\rho)$ is uniformly bounded for $\rho\in\mathbb R$, we have $$\lim_{\rho\to \infty}\phi(\rho) = b,\quad \lim_{\rho\to \infty} \phi'(\rho) = 0,\text{ for some }b>0.$$
So we have $a\le \phi\le b$. However, the volume of $(X_\infty,g_\infty)$ is given by
$$Vol(X_\infty,g_\infty) =C(n) \int_{-\infty}^{\infty}(\phi)^{n-1}\phi'd\rho = C(n)\bk{(\lim_{\rho\to\infty}\phi(\rho))^n - a^n},$$ and we know $Vol(X_\infty,g_\infty)$ is unbounded, hence $\lim_{\rho\to \infty}\phi(\rho)$ is not bounded, and we get a contradiction.


Suppose $\nu<0$, then for large $\phi$, $F(\phi)$ is dominated by $\nu \phi^{1-n}e^{\mu \phi}<0$, and this implies $F(\phi)$ has another zero $b>a$, which contradicts the unboundedness of the volume of $(X_\infty,g_\infty)$ as before. If $\nu>0$, $F$ is controlled by the term $\nu \phi^{1-n} e^{\mu \phi}>0$ when $\phi$ is large, so there is no second zero $b$ of $F$, and $F>0$ on $\phi\in (a,\infty)$, $\phi(\rho)\to \infty$ as $\rho$ converges to a maximal value $\rho_0<\infty$.

For $\phi$ large enough, we have $\phi'\ge c e^{2\mu\phi/3}$ for some small constant $c = c(\nu)>0$, integrating over $[\rho,\rho_0)$, we have $$e^{\mu\phi(\rho)}\le \frac{1}{c\mu (\rho_0 - \rho)^{3/2}},$$and hence for $\phi$ large
$$u_\infty'' = \phi'\le \frac{2\nu}{a^{n-1}} \frac{1}{c\mu (\rho_0 - \rho)^{3/2}},  $$ then the integral 
$$\int_0^{\rho_0} \sqrt{u_\infty''}d\rho\le C \int_0^{\rho_0} \frac{1}{(\rho_0 - \rho)^{3/4}}d\rho<\infty$$ contradicting the completeness of the metric $g_\infty$ on $X_\infty$.
\end{proof}
Hence from \eqref{equation solution} we know that the solution $\phi$ exists for all $\rho\in (-\infty, \infty)$ since the leading term on RHS of \eqref{equation solution} is the linear $\phi/\mu$ when $\phi$ is large and this implies $X_\infty$ is the line bundle $\mathcal O_{\mathbb{CP}^{n-1}}(-1)$, and from \eqref{determine mu} we have
\begin{equation*}
 \frac{a}{\mu} - \frac{\mu - 1}{\mu^{n+1}}\sum_{j=0}^{n-1}\frac{n!}{j!}\mu^j a^{j+1-n} = 0,
\end{equation*}
which must have a positive root $\mu=\mu(n)$ for the given $a= n-1$ by the intermediate value theorem, and for this root $\mu$, the solution $\phi$ to \eqref{equation solution} defines a complete K\"ahler Ricci soliton, which must be one of  the FIK solutions constructed in \cite{FIK}.


\begin{thebibliography}{100}
\bibitem{Calabi}Calabi, E. {\em Extremal K\"ahler metrics}, in Seminar on Differential Geometry, pp. 259-290, Ann. of Math. Stud., 102, Princeton Univ. Press, Princeton, N.J., 1982

\bibitem{C} Cao, H. {\em Deformation of K\"ahler metrics to K\"ahler-Einstein metrics on compact K\"ahler manifolds},
Invent. Math. 81 (1985), no. 2, 359-372

\bibitem{C1}Cao, H. {\em Existence of gradient K\"ahler-Ricci solitons}, Elliptic and parabolic methods in geometry (Minneapolis, MN, 1994), 1-16, A K Peters, Wellesley, MA, 1996.


\bibitem{CZ} Cao, H. and Zhou, D. {\em On complete gradient shrinking Ricci solitons}. J. Differential Geom. 85.2 (2010): 175-186.

\bibitem{CLN} Chen, X., Lu, P. and Tian, G. {\em A note on uniformization of Riemann surfaces by Ricci flow.} Proc. of the A.M.S. (2006): 3391-3393.


\bibitem{ChWa}Chen, X. and Wang, B., {\em Space of Ricci flows (II)}. arXiv:1405.6797 (2014).

\bibitem{EMT} Enders, J., Mueller R. and Topping, P. {\em On Type I Singularities in Ricci flow},  Comm. Anal. Geom. 19 (2011), no. 5, 905-922.

\bibitem{FIK}Feldman, M., Ilmanen, T. and Knopf, D. {\em Rotationally symmetric shrinking and expanding gradient K\"ahler-Ricci solitons}, J. Differential Geom. 65 (2003), no. 2, 169 - 209.

\bibitem{Fo}Fong, T. {\em On the Collapsing Rate of the K\"ahler Ricci Flow with Finite-Time Singularity}. J. Geom. Ana. (2011), 1-10.

\bibitem{H}Hamilton, R. {\em Three manifolds with positive Ricci curvature}, J. Differential Geom.17 (1982), no.2, 255 - 306.


\bibitem{Koiso}Koiso, N. {\em On rotationally symmetric Hamilton's equation for K\"ahler-Einstein metrics}, Recent
topics in differential and analytic geometry, 327-337, Adv. Stud. Pure Math., 18-I, Academic
Press, Boston, MA, 1990

\bibitem{Ma}Maximo, D. {\em On the blow-up of four-dimensional Ricci flow singularities}, J. Reine Angew. Math. 692 (2014), 153-171. 

 \bibitem{Nab} Naber, A. {\em  Noncompact shrinking 4-solitons with nonnegative curvature}, J. Reine Angew. Math. 645 (2010), 125-153.
 
 
\bibitem{Pe} Perelman, G. {\em The entropy formula for the Ricci flow and its geometric applications}, arXiv: math/0211159.

\bibitem{PS} Phong, D.H. and Sturm, J. {\em On stability and the convergence of the K\"ahler-Ricci flow},
J. Differential Geom. 72 (2006), no. 1, 149 - 168

\bibitem{PSS} Phong, D. H., Sesum, N. and Sturm, J. {\em Multiplier ideal sheaves and the K\"ahler-Ricci flow}, Comm.
Anal. Geom. 15 (2007), no. 3, 613 - 632

\bibitem{PSSW1}Phong, D.H., Sturm, J., Song, J. and Weinkove, B. {\em The K\"ahler-Ricci flow with positive bisectional curvature}, Invent. Math. 173 (2008), no. 3, 651-665

\bibitem{PSSW2}Phong, D.H., Sturm, J., Song, J. and Weinkove, B. {\em The K\"ahler-Ricci flow and the $\bar\partial$ operator on vector fields},  J. Differential Geom. 81 (2009), no. 3, 631-647

\bibitem{PSSW0} Phong, D. H., Song, J., Sturm, J. and Weinkove, B. {\em On the convergence of the modified K\"ahler-Ricci flow and solitons} Comment. Math. Helv. 86 (2011), no. 1, 91--112

\bibitem{SeTi}Sesum, N., and Tian, G. {\em Bounding scalar curvature and diameter along the K\"ahler Ricci flow (after Perelman)}. J. Inst. Math. Jussieu 7 (2008), no. 3, 575 - 587.

\bibitem{Song1}Song, J. {\em Some type I solutions of Ricci flow with rotational symmetry}, to appear in I.M.R.N., arXiv:1203.2688

\bibitem{ST1}Song, J. and Tian, G. {\em The K\"ahler Ricci flow on surfaces of positive Kodaira dimension}, Invent. math. 170.3 (2007): 609-653.

\bibitem{ST0}  Song, J. and Tian, G. {\em Canonical measures and K\"ahler-Ricci flow}, J . Amer. Math. Soc. 25 (2012),
303-353.

\bibitem{ST3} Song, J. and Tian, G. {\em The K\"ahler-Ricci flow through singularities}, arXiv:0909.4898.

\bibitem{SWe0}Song, J. and Weinkove, B. {\em The K\"ahler-Ricci flow on Hirzebruch surfaces}, J. Reine Angew. Math. 659 (2011), 141-168.


\bibitem{SWe} Song, J. and Weinkove, B. {\em Contracting exceptional divisors by the K\"ahler-Ricci flow}, Duke Math. J. 162 (2013), no. 2, 367-415.

\bibitem{SWe2}Song, J. and Weinkove, B. {\em Contracting exceptional divisors by the K\"ahler-Ricci flow II}, Proc. Lond. Math. Soc. (3) 108 (2014), no. 6, 1529-1561.

\bibitem{TiZz} Tian, G. and Zhang, Z. {\em On the K\"ahler-Ricci flow on projective manifolds of general type}, Chinese Ann.
Math. Ser. B 27 (2006), no. 2, 179-192.

 \bibitem{WZ} Wang, X. and Zhu, X. {\em K\"ahler-Ricci solitons on toric manifolds with positive first Chern class},
Adv. Math. 188 (2004), no. 1, 87-103.

\bibitem{Y1}Yau, S.-T. {\em On the Ricci curvature of a compact K\"ahler manifold and the complex Monge-Amp\`ere
equation}, I, Comm. Pure Appl. Math. 31 (1978), 339-411.

\bibitem{Zhux}Zhu, X. {\em  K\"ahler-Ricci flow on a toric manifold with positive first Chern class},
arXiv:math/0703486.

\end{thebibliography}
\end{document}